\documentclass[12pt]{article}
\usepackage[latin1]{inputenc}

\usepackage{fullpage}
\usepackage{amsfonts}
\usepackage{amssymb}
\usepackage{amsmath}
\usepackage{color}
\usepackage[mathscr]{euscript}

\usepackage{mathtools}
\usepackage{stmaryrd}
\usepackage{wasysym}

\newtheorem{theorem}{Theorem}
\newtheorem{definition}[theorem]{Definition}

\newtheorem{corollary}[theorem]{Corollary}
\newenvironment{proof}{\small{\bf Proof.}}%
  {\hfill$\Box$\normalsize\bigskip}

{\normalsize}

{\normalsize}

{\normalsize}

\newtheorem{proposition}[theorem]{Proposition}

\newcommand{\eqsepv}{\; , \enspace}       
\newcommand{\eqfinv}{\; ,}                
\newcommand{\eqfinp}{\; .}

\renewcommand{\bar}{\overline}

\newcommand{\mtext}[1]{\,\mbox{#1}\,} 
\newcommand{\module}[1]{| #1 |}
\newcommand{\norm}[1]{\|#1\|}

\newcommand{\sequence}[2]{\left\{#1\right\}_{#2}}           
\newcommand{\proscal}[2]{\left\langle#1\:,#2\right\rangle}  

\newcommand{\np}[1]{(#1)}                                   
\newcommand{\bp}[1]{\big(#1\big)}                           
\newcommand{\Bp}[1]{\Big(#1\Big)}                           
\newcommand{\bgp}[1]{\bigg(#1\bigg)}                        

\newcommand{\Bc}[1]{\Big[#1\Big]}                           

\newcommand{\na}[1]{\{#1\}}                                 
\newcommand{\ba}[1]{\big\{#1\big\}}                         

\newcommand{\RR}{{\mathbb R}} 
\newcommand{\NN}{{\mathbb N}} 

\newcommand{\AAA}{{\mathbb A}} 
\newcommand{\BB}{{\mathbb B}} 

\newcommand{\defset}[2]{\left\{#1\:\left|\:#2\right.\right\}}

\newcommand{\uncertain}{w}
\newcommand{\Uncertain}{W}
\newcommand{\UNCERTAIN}{{\mathbb W}}

\newcommand{\scenario}{s}
\newcommand{\Scenario}{S}
\newcommand{\SCENARIO}{\mathbb{S}}

\def\stackops#1#2#3{%
  \mathrel{\vbox{\offinterlineskip\ialign{%
    \hfil##\hfil\cr
    $#1$\cr
    \noalign{\kern#3}
    $#2$\cr}}}}

\def\plusdot{\stackops{\cdot}{+}{-2.5ex}}

\newcommand{\LowPlus}{\plusdot}  
\newcommand{\UppPlus}{\dotplus}

\newcommand{\HILBERT}{{\mathbb V}}
\newcommand{\Hilbert}{V}
\newcommand{\hilbert}{v}

\newcommand{\PRIMAL}{{\mathbb X}}
\newcommand{\Primal}{X}
\newcommand{\primal}{x}
\newcommand{\DUAL}{{\mathbb Y}}
\newcommand{\Dual}{Y}
\newcommand{\dual}{y}
\newcommand{\ExtendedReals}{[-\infty,+\infty]}
\newcommand{\barRR}{\overline{\mathbb R}}
\newcommand{\Convex}{C}
\newcommand{\cardinal}[1]{\module{#1}}

\newcommand{\dom}{\mathrm{dom}}
\newcommand{\barriercone}{\mathrm{bar}}
\newcommand{\closedspan}{\overline{\mathrm{span}}}
\newcommand{\conicalhull}{\mathrm{cone}}
\newcommand{\continuitypoints}{\mathrm{cont}}
\newcommand{\convexhull}{\mathrm{co}}
\newcommand{\closedconvexhull}{\overline{\mathrm{co}}}
\newcommand{\interior}{\mathrm{int}}
\newcommand{\linearspan}{\mathrm{span}}

\newcommand{\fonctionprimal}{f} 
\newcommand{\fonctiondual}{g} 
\newcommand{\fonctionprimalbis}{h} 
\newcommand{\fonctionuncertain}{h} 
\newcommand{\InfimalPostcomposition}[2]{#1\rhd#2} 
 
\newcommand{\lzero}{\ell_0}
\newcommand{\pseudonormlzero}{$l_0$~pseudonorm}
\newcommand{\LevelSet}[2]{#1^{\leq #2}}

\newcommand{\coupling}{c}
\newcommand{\Capra}{Caprac} 
\newcommand{\couplingCAPRA}{\cent} 
\newcommand{\LFM}[1]{#1^{\star}}
\newcommand{\LFMr}[1]{#1^{\star'}}
\newcommand{\LFMbi}[1]{#1^{\star\star'}}
\newcommand{\minusLFM}[1]{#1^{(-\star)}}
\newcommand{\minusLFMr}[1]{#1^{(-\star)'}}
\newcommand{\minusLFMbi}[1]{#1^{(-\star)(-\star)'}}
\newcommand{\SFM}[2]{#1^{#2}}
\newcommand{\SFMbi}[2]{#1^{#2{#2}'}}
\newcommand{\triplenorm}[1]{||| #1 |||}

\newcommand{\SymmetricGaugeNorm}[2]{\norm{#2}_{(#1)}^{\mathrm{sgn}}}
\newcommand{\SupportNorm}[2]{\norm{#2}_{(#1)}^{\mathrm{sn}}}

\newcommand{\Support}[1]{\mathrm{supp}\np{#1}}
\newcommand{\SPHERE}{S} 
\newcommand{\BALL}{B}

\newcommand{\normalized}{n}

\newcommand{\matrice}{A}
\newcommand{\Subspace}{\HILBERT}
\newcommand{\scenariobis}{s'}

\renewcommand{\scenario}{j}
\renewcommand{\scenariobis}{j'}
\renewcommand{\Scenario}{J}
\renewcommand{\SCENARIO}{\mathbb{J}}

\title{Lower Bound Convex Programs for\\ 
Exact Sparse Optimization}

\author{Jean-Philippe Chancelier and 
Michel De Lara, \\ CERMICS, \'Ecole des Ponts ParisTech}

\begin{document}

\maketitle

\begin{abstract}
In exact sparse optimization problems on~$\RR^d$
(also known as sparsity constrained problems), 
one looks for solution that have few nonzero components. 
In this paper, we consider problems where sparsity is exactly measured either by
the nonconvex \pseudonormlzero\ (and not by substitute penalty terms)
or by the belonging of the solution to a finite union of subsets.
Due to the combinatorial nature of the sparsity constraint,
such problems do not generally display convexity properties, 
even if the criterion to minimize is convex. 
In the most common approach to tackle them, one replaces the sparsity constraint
by a convex penalty term, supposed to induce sparsity.
Thus doing, one loses the original exact sparse optimization problem,
but gains convexity. 
However, by doing so, it is not clear that one obtains a lower bound
of the original exact sparse optimization problem.
In this paper, 
we propose another approach, where we lose convexity but where we gain at keeping
 the original exact sparse optimization formulation,
by displaying lower bound convex minimization programs.
For this purpose, we introduce suitable conjugacies, 
induced by a novel class of one-sided linear couplings.
Thus equipped, we present a systematic way to design norms
and lower bound convex minimization programs over their unit ball. 
The family of norms that we display encompasses 
most of the sparsity inducing norms used in machine learning.
Therefore, our approach provides foundation and interpretation
for their use.
\end{abstract}



\section{Introduction}

In exact sparse optimization problems on~$\RR^d$
(also known as sparsity constrained problems), 
one looks for solution that have few nonzero components. 
The \emph{counting function}, also called \emph{cardinality function}
or \emph{\pseudonormlzero}, 
counts the number of nonzero components of a vector in~$\RR^d$.
It is well-known that the \pseudonormlzero\ is 
lower semi continuous but is not convex.
As a consequence, a minimization problem under the constraint that
the \pseudonormlzero\ is less than a given integer 
is not convex in general. 
Then, it is common practice to \emph{replace} the nonconvex sparsity constraint
\emph{by substitute (convex) penalty terms}, supposed to induce sparsity.
By doing so, on the one hand, one gains convexity
and benefits of duality tools with the Fenchel conjugacy.
However, on the other hand, it is not clear that one obtains a lower bound
of the original exact sparse optimization problem.

In this paper, we consider exact sparse optimization problems,
that is, problems with combinatorial sparsity constraint.
More precisely, we focus on
problems where \emph{sparsity is exactly measured} either by
the nonconvex \pseudonormlzero\ (and not by substitute penalty terms)
or by the belonging of the solution to a finite union of given subsets.
Our main contribution is to provide a systematic way to 
design norms, and associated convex programs 
that are lower bounds for the original 
exact sparse optimization problem.

The paper is organized as follows.
In Sect.~\ref{One-sided_linear_couplings_and_lower_bound_convex_programs},
we recall the definition
and properties of so-called one-sided linear couplings,
introduced in the companion paper \cite{Chancelier-DeLara:2019b},
and we show how to use them to obtain 
concave maximization/convex minimization problems
that are lower bounds of a given optimization problem.
In Sect.~\ref{Dual_problems_for_exact_sparse_optimization_problems},
we consider minimization problems under the constraint that
the \pseudonormlzero\ is less than a given integer.
To provide a lower bound, we make use of a suitable conjugacy
(not the Fenchel one) induced by the so-called coupling \Capra,
introduced in~\cite{Chancelier-DeLara:2019b}.
We obtain a concave maximization program as lower bound and,
under a mild assumption, it coincides with a convex minimization program
on the unit ball of the so-called $k$-support norm. 
In Sect.~\ref{One-sided_convex_couplings_and_generalized_sparse_optimization},
we consider generalized exact sparse optimization problems.
These are minimization problems under the constraint that
the solution belongs to a finite union of given subsets.
We present a systematic way to design norms
and lower bound convex minimization programs over their unit ball.

\section{One-sided linear couplings and lower bound convex programs}
\label{One-sided_linear_couplings_and_lower_bound_convex_programs}

In~\S\ref{One-sided_linear_couplings}, we recall the definition
and properties of \emph{one-sided linear couplings}.
In~\S\ref{Lower_bound_convex_programs},
we show how to use them to obtain 
concave maximization/convex minimization problems
that are lower bounds of a given optimization problem.

\subsection{One-sided linear couplings and conjugacies}
\label{One-sided_linear_couplings}

The material here is mostly taken from~\cite{Chancelier-DeLara:2019b}.
Basic recalls and notations used in analysis can be found
in~\S\ref{Background_on_sets_and_functions}.

\subsubsection{Background on couplings and conjugacies}

We review general concepts and notations, then we focus 
on the special case of the Fenchel conjugacy.
We denote \( \barRR=\ExtendedReals \). 
Background on J.~J. Moreau lower and upper additions can be found 
in~\S\ref{Moreau_lower_and_upper_additions}.

\subsubsection*{The general case}

Let be given two sets $\PRIMAL$ (``primal''), $\DUAL$ (``dual''), together 
with a \emph{coupling} function
\begin{equation}
  \coupling : \PRIMAL \times \DUAL \to \barRR 
\eqfinp 
\end{equation}

With any coupling, we associate \emph{conjugacies} 
from \( \barRR^\PRIMAL \) to \( \barRR^\DUAL \) 
and from \( \barRR^\DUAL \) to \( \barRR^\PRIMAL \) 
as follows.

\begin{definition}
  The \emph{$\coupling$-Fenchel-Moreau conjugate} of a 
  function \( \fonctionprimal : \PRIMAL  \to \barRR \), 
  with respect to the coupling~$\coupling$, is
  the function \( \SFM{\fonctionprimal}{\coupling} : \DUAL  \to \barRR \) 
  defined by
  \begin{equation}
    \SFM{\fonctionprimal}{\coupling}\np{\dual} = 
    \sup_{\primal \in \PRIMAL} \Bp{ \coupling\np{\primal,\dual} 
      \LowPlus \bp{ -\fonctionprimal\np{\primal} } } 
    \eqsepv \forall \dual \in \DUAL
    \eqfinp
    \label{eq:Fenchel-Moreau_conjugate}
  \end{equation}
With the coupling $\coupling$, we associate 
the \emph{reverse coupling~$\coupling'$} defined by 
\begin{equation}
  \coupling': \DUAL \times \PRIMAL \to \barRR 
\eqsepv
\coupling'\np{\dual,\primal}= \coupling\np{\primal,\dual} 
\eqsepv
\forall \np{\dual,\primal} \in \DUAL \times \PRIMAL
\eqfinp
    \label{eq:reverse_coupling}
\end{equation}
  The \emph{$\coupling'$-Fenchel-Moreau conjugate} of a 
  function \( \fonctiondual : \DUAL \to \barRR \), 
  with respect to the coupling~$\coupling'$, is
  the function \( \SFM{\fonctiondual}{\coupling'} : \PRIMAL \to \barRR \) 
  defined by
  \begin{equation}
    \SFM{\fonctiondual}{\coupling'}\np{\primal} = 
    \sup_{ \dual \in \DUAL } \Bp{ \coupling\np{\primal,\dual} 
      \LowPlus \bp{ -\fonctiondual\np{\dual} } } 
    \eqsepv \forall \primal \in \PRIMAL 
    \eqfinp
    \label{eq:Fenchel-Moreau_reverse_conjugate}
  \end{equation}
  The \emph{$\coupling$-Fenchel-Moreau biconjugate} of a 
  function \( \fonctionprimal : \PRIMAL  \to \barRR \), 
  with respect to the coupling~$\coupling$, is
  the function \( \SFMbi{\fonctionprimal}{\coupling} : \PRIMAL \to \barRR \) 
  defined by
  \begin{equation}
    \SFMbi{\fonctionprimal}{\coupling}\np{\primal} = 
    \bp{\SFM{\fonctionprimal}{\coupling}}^{\coupling'} \np{\primal} = 
    \sup_{ \dual \in \DUAL } \Bp{ \coupling\np{\primal,\dual} 
      \LowPlus \bp{ -\SFM{\fonctionprimal}{\coupling}\np{\dual} } } 
    \eqsepv \forall \primal \in \PRIMAL 
    \eqfinp
    \label{eq:Fenchel-Moreau_biconjugate}
  \end{equation}
%
%
\end{definition}

  \begin{subequations}
For any coupling~$\coupling$, 
\begin{itemize}
\item 
the biconjugate of a 
  function \( \fonctionprimal : \PRIMAL  \to \barRR \) satisfies
\begin{equation}
  \SFMbi{\fonctionprimal}{\coupling}\np{\primal}
  \leq \fonctionprimal\np{\primal}
  \eqsepv \forall \primal \in \PRIMAL 
  \eqfinv
  \label{eq:galois-cor}
\end{equation}
\item 
for any couple of functions \( \fonctionprimal : \PRIMAL  \to \barRR \) 
and \( \fonctionprimalbis : \PRIMAL \to \barRR \), we have the inequality
\begin{equation}
     \sup_{\dual \in \DUAL} 
    \Bp{ 
\bp{-\SFM{\fonctionprimal}{\coupling}\np{\dual} } 
      \LowPlus 
\bp{-\SFM{\fonctionprimalbis}{-\coupling}\np{\dual} }
}
    \leq 
    \inf_{\primal \in \PRIMAL} 
\Bp{ 
\fonctionprimal\np{\primal} 
      \UppPlus 
\fonctionprimalbis\np{\primal} 
} 
\eqfinv
  \label{eq:dual_problem_inequality}
\end{equation}
where the \emph{$\np{-\coupling}$-Fenchel-Moreau conjugate} is given by
  \begin{equation}
    \SFM{\fonctionprimalbis}{-\coupling}\np{\dual} = 
    \sup_{\primal \in \PRIMAL} \Bp{ \bp{ -\coupling\np{\primal,\dual} }
      \LowPlus \bp{ -\fonctionprimalbis\np{\primal} } } 
    \eqsepv \forall \dual \in \DUAL
    \eqfinv
    \label{eq:minus_Fenchel-Moreau_conjugate}
  \end{equation}
\item 
for any function \( \fonctionprimal : \PRIMAL  \to \barRR \) 
and subset \( \Primal \subset \PRIMAL \), we have the inequality
\begin{equation}
      \sup_{\dual \in \DUAL} 
    \Bp{ 
\bp{-\SFM{\fonctionprimal}{\coupling}\np{\dual} } 
      \LowPlus 
\bp{-\SFM{\delta_{\Primal}}{-\coupling}\np{\dual} }
}
    \leq 
    \inf_{\primal \in \PRIMAL} 
\Bp{ 
\fonctionprimal\np{\primal} 
      \UppPlus 
\delta_{\Primal} \np{\primal} } 
= \inf_{\primal \in \Primal} \fonctionprimal\np{\primal} 
\eqfinp
  \label{eq:dual_problem_inequality_constraints}
\end{equation}
\end{itemize}
  \end{subequations}

\subsubsection*{The Fenchel conjugacy}

When the sets $\PRIMAL$ and $\DUAL$ are vector spaces 
equipped with a bilinear form \( \proscal{}{} \),
the corresponding conjugacy is the classical 
\emph{Fenchel conjugacy}.
For any functions \( \fonctionprimal : \PRIMAL  \to \barRR \)
and \( \fonctiondual : \DUAL \to \barRR \), we denote
\begin{subequations}
\begin{align}
    \LFM{\fonctionprimal}\np{\dual} 
&= 
    \sup_{\primal \in \PRIMAL} \Bp{ \proscal{\primal}{\dual} 
      \LowPlus \bp{ -\fonctionprimal\np{\primal} } } 
    \eqsepv \forall \dual \in \DUAL
    \eqfinv
    \label{eq:Fenchel_conjugate}
  \\
    \LFMr{\fonctiondual}\np{\primal} 
&= 
    \sup_{ \dual \in \DUAL } \Bp{ \proscal{\primal}{\dual} 
      \LowPlus \bp{ -\fonctiondual\np{\dual} } } 
    \eqsepv \forall \primal \in \PRIMAL 
    \label{eq:Fenchel_conjugate_reverse}
\\
    \LFMbi{\fonctionprimal}\np{\primal} 
&= 
    \sup_{\dual \in \DUAL} \Bp{ \proscal{\primal}{\dual} 
      \LowPlus \bp{ -\LFM{\fonctionprimal}\np{\dual} } } 
    \eqsepv \forall \primal \in \PRIMAL
    \eqfinp
    \label{eq:Fenchel_biconjugate}
\end{align}
\end{subequations}
Due to the presence of the coupling \( \np{-\coupling} \) 
in the Inequality~\eqref{eq:dual_problem_inequality},
we also introduce\footnote{%
In convex analysis, one does not use the notations below, but rather uses 
\( \fonctionprimal^{\lor}\np{\primal} =\fonctionprimal\np{-\primal} \), 
for all \( \primal \in \PRIMAL \), 
and
\( \fonctiondual^{\lor}\np{\dual}=\fonctiondual\np{-\dual} \), 
for all \( \dual \in \DUAL \).
The connection between both notations is given by 
\(  \minusLFM{\fonctionprimal} =
\LFM{ \bp{\fonctionprimal^{\lor}} } = 
\bp{\LFM{\fonctionprimal}}^{\lor} \).
}
\begin{subequations}
\begin{align}
    \minusLFM{\fonctionprimal}\np{\dual} 
&= 
    \sup_{\primal \in \PRIMAL} \Bp{ -\proscal{\primal}{\dual} 
      \LowPlus \bp{ -\fonctionprimal\np{\primal} } } 
= \LFM{\fonctionprimal}\np{-\dual} 
    \eqsepv \forall \dual \in \DUAL
    \eqfinv
    \label{eq:minusFenchel_conjugate}
  \\
    \minusLFMr{\fonctiondual}\np{\primal} 
&= 
    \sup_{ \dual \in \DUAL } \Bp{ -\proscal{\primal}{\dual} 
      \LowPlus \bp{ -\fonctiondual\np{\dual} } } 
= \LFMr{\fonctiondual}\np{-\primal} 
    \eqsepv \forall \primal \in \PRIMAL 
    \label{eq:minusFenchel_conjugate_reverse}
\\
    \minusLFMbi{\fonctionprimal}\np{\primal} 
&= 
    \sup_{\dual \in \DUAL} \Bp{ -\proscal{\primal}{\dual} 
      \LowPlus \bp{ -\minusLFM{\fonctionprimal}\np{\dual} } } 
=\LFMbi{\fonctionprimal}\np{\primal}
    \eqsepv \forall \primal \in \PRIMAL
    \eqfinp
    \label{eq:minusFenchel_biconjugate}
\end{align}
\end{subequations}



When the two vector spaces $\PRIMAL$ and $\DUAL$ are \emph{paired}
in the sense of convex analysis\footnote{That is,
$\PRIMAL$ and $\DUAL$ are equipped with 
a bilinear form \( \proscal{}{} \), 
and locally convex topologies
that are compatible in the sense
that the continuous linear forms on~$\PRIMAL$
are the functions 
\( \primal \in \PRIMAL \mapsto \proscal{\primal}{\dual} \),
for all \( \dual \in \DUAL \),
and 
that the continuous linear forms on~$\DUAL$
are the functions 
\( \dual \in \DUAL \mapsto \proscal{\primal}{\dual} \),
for all \( \primal \in \PRIMAL \)},
Fenchel conjugates are convex 
\emph{lower semi continuous (lsc)} functions,
and their opposites are concave
\emph{upper semi continuous (usc)} functions.

\subsubsection{One-sided linear couplings and conjugacies}

In the companion paper \cite{Chancelier-DeLara:2019b}, 
we have introduced and studied a novel
family of couplings defined as follows. 

Let $\UNCERTAIN$ and $\PRIMAL$ be two sets and
\( \theta: \UNCERTAIN \to \PRIMAL \) be a mapping.
We recall the definition \cite[p.~214]{Bauschke-Combettes:2017}
of the \emph{infimal postcomposition} 
\( \bp{\InfimalPostcomposition{\theta}{\fonctionuncertain}}:
\PRIMAL \to \barRR \) of 
a function \( \fonctionuncertain : \UNCERTAIN \to \barRR \): 
  \begin{equation}
\bp{\InfimalPostcomposition{\theta}{\fonctionuncertain}}\np{\primal}
=
\inf\defset{\fonctionuncertain\np{\uncertain}}{%
\uncertain\in\UNCERTAIN \eqsepv \theta\np{\uncertain}=\primal}
\eqsepv \forall \primal \in \PRIMAL 
  \eqfinv
\label{eq:InfimalPostcomposition}
  \end{equation}
with the convention that \( \inf \emptyset = +\infty \)
(and with the consequence that 
\( \theta: \UNCERTAIN \to \PRIMAL \) need not be defined
on all~$\UNCERTAIN$, but only on
the effective domain~\( \dom\fonctionuncertain = 
\defset{\uncertain\in\UNCERTAIN}{\fonctionuncertain\np{\uncertain}<+\infty} \) of the 
function \( \fonctionuncertain : \UNCERTAIN \to \barRR \)).

\begin{definition}
Let $\PRIMAL$ and $\DUAL$ be two vector spaces 
equipped with a bilinear form \( \proscal{}{} \).
Let $\UNCERTAIN$ be a set and
\( \theta: \UNCERTAIN \to \PRIMAL \) a mapping.
  We define the \emph{one-sided linear coupling} $\coupling_{\theta}$
between $\UNCERTAIN$ and $\DUAL$ by
\begin{equation}
\coupling_{\theta} : \UNCERTAIN \times \DUAL \to \barRR
\eqsepv 
\coupling_{\theta}\np{\uncertain, \dual} = 
\proscal{\theta\np{\uncertain}}{\dual} 
\eqsepv \forall \uncertain \in \UNCERTAIN
\eqsepv \forall \dual \in \DUAL
\eqfinp 
\label{eq:one-sided_linear_coupling}
\end{equation}
\end{definition}

Here are expressions for the $\coupling_{\theta}$-conjugates 
in function of the Fenchel conjugate.
The proof of Proposition~\ref{pr:one-sided_linear_Fenchel-Moreau_conjugate}
can be found in \cite{Chancelier-DeLara:2019b}.

\begin{subequations}
\begin{proposition}
  For any function \( \fonctiondual : \DUAL \to \barRR \), 
  the $\coupling_{\theta}'$-Fenchel-Moreau conjugate is given by 
  \begin{equation}
\SFM{\fonctiondual}{\coupling_{\theta}'}=
 \LFM{ \fonctiondual } \circ \theta
\eqfinp
\label{eq:one-sided_linear_c'-Fenchel-Moreau_conjugate}
\end{equation}
  For any function \( \fonctionuncertain : \UNCERTAIN \to \barRR \), 
  the $\coupling_{\theta}$-Fenchel-Moreau conjugate is given by 
   \begin{equation}
\SFM{\fonctionuncertain}{\coupling_{\theta}}=
 \LFM{ \bp{\InfimalPostcomposition{\theta}{\fonctionuncertain}} }
\eqfinv 
\label{eq:one-sided_linear_Fenchel-Moreau_conjugate}
  \end{equation}
and the $\coupling_{\theta}$-Fenchel-Moreau biconjugate 
is given by
\begin{equation}
  \SFMbi{\fonctionuncertain}{\coupling_{\theta}}
= 
\SFM{ \bp{ \SFM{\fonctionuncertain}{\coupling_{\theta}} } }{\star} 
\circ \theta
=
 \LFMbi{ \bp{\InfimalPostcomposition{\theta}{\fonctionuncertain}} }
\circ \theta
\eqfinp
\label{eq:one-sided_linear_Fenchel-Moreau_biconjugate}
\end{equation}
For any subset \( \Uncertain \subset \UNCERTAIN \), 
the $\np{-\coupling_{\theta}}$-Fenchel-Moreau conjugate 
of the characteristic function of~$\Uncertain$ 
is given by the following support function
\begin{equation}
\SFM{ \delta_{\Uncertain} }{-\coupling_{\theta}}
= \sigma_{ -\theta\np{\Uncertain} } 
\eqfinp
\label{eq:one-sided_linear_Fenchel-Moreau_characteristic}
\end{equation}
  \label{pr:one-sided_linear_Fenchel-Moreau_conjugate}
\end{proposition}
\end{subequations}

\subsection{Lower bound convex programs}
\label{Lower_bound_convex_programs}

To illustrate how we can obtain lower bounds with one-sided linear couplings,
we start with general problems of the form
\begin{equation}
\inf_{\uncertain \in \Uncertain} \fonctionuncertain\np{\uncertain} 
\eqfinv
\label{eq:optimization}
\end{equation}
where \( \fonctionuncertain : \UNCERTAIN \to \barRR \)
and \( \Uncertain \subset \UNCERTAIN \)
(we can always replace the subset~\( \Uncertain \)
by \( \dom\fonctionuncertain \cap \Uncertain \)).

\begin{subequations}
\begin{proposition}
Let $\PRIMAL$ and $\DUAL$ be two vector spaces 
equipped with a bilinear form \( \proscal{}{} \).
Let \( \UNCERTAIN \) be a set.
For any function \( \fonctionuncertain : \UNCERTAIN \to \barRR \),
nonempty set \( \Uncertain \subset \UNCERTAIN \)
and mapping \( \theta: \Uncertain \to \PRIMAL \) ,
we have the following lower bound 
\begin{equation}
      \sup_{\dual \in \DUAL} 
    \Bp{ 
\bp{-\LFM{ \bp{ \InfimalPostcomposition{\theta}{\fonctionuncertain} } }\np{\dual} } 
      \LowPlus 
\bp{-\sigma_{ -\theta\np{\Uncertain} } \np{\dual} }
}
    \leq 
\inf_{\uncertain \in \Uncertain} \fonctionuncertain\np{\uncertain} 
\eqfinp
\label{eq:lower_bound_concave_program}
\end{equation}
\label{pr:lower_bound_concave_program}
\end{proposition}

\begin{proof}
As \( \inf_{\uncertain \in \Uncertain} \fonctionuncertain\np{\uncertain} =
\inf_{\uncertain \in \UNCERTAIN} \bp{ \fonctionuncertain\np{\uncertain} 
\UppPlus \delta_{\Uncertain} } \), 
it suffices to use the Inequality~\eqref{eq:dual_problem_inequality},
with the $\coupling_{\theta}$-Fenchel-Moreau conjugate of
the function~\( \fonctionuncertain \) 
given by~\eqref{eq:one-sided_linear_Fenchel-Moreau_conjugate}
and the $\np{-\coupling_{\theta}}$-Fenchel-Moreau conjugate of
the characteristic function~\(  \delta_{\Uncertain} \) 
given by~\eqref{eq:one-sided_linear_Fenchel-Moreau_characteristic}.
\end{proof}

When $\PRIMAL$ and $\DUAL$ are two paired vector spaces,
the dual problem to the left hand side 
of~\eqref{eq:lower_bound_concave_program}
consists in the maximization of a usc concave function.
\medskip

When $\PRIMAL$ and $\DUAL$ for a dual system (see~\S\ref{Dual_system}),
recall that a set \( \Primal\subset \PRIMAL\) 
is said to be \emph{weakly bounded}
if \( \sup_{\primal\in\Primal} \proscal{\primal}{\dual} < +\infty \)
for all \( \dual \in \DUAL \) (see~\eqref{eq:weakly_bounded} 
in Definition~\ref{de:dual_system}). 
Now, when the primal and dual spaces in 
Proposition~\ref{pr:lower_bound_concave_program}
are a Hilbert space, we provide conditions for the lower bound, to 
the left of~\eqref{eq:lower_bound_concave_program},
to display an alternative primal expression as the minimization of 
a lsc convex function on a weakly bounded and closed convex set.

\begin{corollary}
Let $\PRIMAL=\DUAL$ be a Hilbert space.
Let \( \UNCERTAIN \) be a set.
Let \( \fonctionuncertain : \UNCERTAIN \to \barRR \) be a function, 
\( \Uncertain \subset \UNCERTAIN \) be a nonempty set
and \( \theta: \Uncertain \to \PRIMAL \) be a mapping.
If
\begin{enumerate}
\item 
the set~$-\theta\np{\Uncertain}$ is weakly bounded, that is, 
the barrier cone of~$-\theta\np{\Uncertain}$ 
in~\eqref{eq:barrier_cone} is the full space, namely 
\( \barriercone\bp{-\theta\np{\Uncertain}} =  \DUAL \),
\item 
the convex lsc function
\( \LFM{ \bp{ \InfimalPostcomposition{\theta}{\fonctionuncertain} } } 
 : \DUAL \to \barRR \) is proper,
\end{enumerate}
then the lower bound, to 
the left of~\eqref{eq:lower_bound_concave_program},
has the alternative primal expression
\begin{equation}
      \sup_{\dual \in \DUAL} 
    \Bp{ 
\bp{-\LFM{ \bp{ \InfimalPostcomposition{\theta}{\fonctionuncertain} } }\np{\dual} } 
      \LowPlus 
\bp{-\sigma_{ -\theta\np{\Uncertain} } \np{\dual} }
}
= 
      \min_{\primal \in \closedconvexhull
\np{-\theta\np{\Uncertain} } }
\LFMbi{ \bp{ \InfimalPostcomposition{\theta}{\fonctionuncertain} } }\np{\primal} 
    \leq 
\inf_{\uncertain \in \Uncertain} \fonctionuncertain\np{\uncertain} 
\eqfinv 
\label{eq:lower_bound_convex_program}  
\end{equation}
where the primal problem to the left consists in
the minimization of 
a lsc convex function on a weakly bounded and closed convex set.
\label{cor:lower_bound_concave_program_convex_program}
\end{corollary}
\end{subequations}

\begin{proof}
We consider the Inequality~\eqref{eq:lower_bound_concave_program}.
On the one hand, the convex lsc function
\( \LFM{ \bp{\InfimalPostcomposition{\normalized}{\fonctionprimal}} } \)
is proper by assumption.
On the other hand, the support function~\( \sigma_{ -\theta\np{\Uncertain} } \)
is convex lsc, and also proper.
Indeed, \( -\infty < \sigma_{ -\theta\np{\Uncertain} } \)
since \( \theta\np{\Uncertain} \not= \emptyset \) by assumption,
and \( \dom\sigma_{ -\theta\np{\Uncertain} }  = \DUAL \)
since \( \barriercone\bp{-\theta\np{\Uncertain}} =  \DUAL \)
by assumption.
As a consequence, the support function~\( \sigma_{ -\theta\np{\Uncertain} } \)
has for effective domain the full space~$\DUAL$,
hence its continuity points are
\( \continuitypoints \bp{ \sigma_{ -\theta\np{\Uncertain} } }= \DUAL \).
As \( \dom\bp{\LFM{ 
      \np{\InfimalPostcomposition{\normalized}{\fonctionprimal}}}} 
\not= \emptyset \), we deduce that 
\( \continuitypoints \bp{ \sigma_{ -\theta\np{\Uncertain} } }
\cap 
\dom\bp{\LFM{ 
      \np{\InfimalPostcomposition{\normalized}{\fonctionprimal}}}} 
= 
\dom\bp{\LFM{ 
      \np{\InfimalPostcomposition{\normalized}{\fonctionprimal}}}} 
\not= \emptyset \).

Thus, the conditions for a Fenchel-Rockafellar equality are satisfied 
\cite[Prop.~15.13]{Bauschke-Combettes:2017} 
and we obtain that 
\begin{equation*}
      \sup_{\dual \in \DUAL} 
    \Bp{ 
\bp{-\LFM{ \bp{ \InfimalPostcomposition{\theta}{\fonctionuncertain} } }\np{\dual} } 
      \LowPlus 
\bp{-\sigma_{ -\theta\np{\Uncertain} } \np{\dual} }
}
= 
      \min_{\primal \in \PRIMAL} \Bp{
\LFMbi{ \bp{ \InfimalPostcomposition{\theta}{\fonctionuncertain} } }\np{\primal} 
\UppPlus \delta_{\closedconvexhull\np{-\theta\np{\Uncertain} } } } 
= 
      \min_{\primal \in \closedconvexhull
\np{-\theta\np{\Uncertain} } }
\LFMbi{ \bp{ \InfimalPostcomposition{\theta}{\fonctionuncertain} } }\np{\primal} 
\eqfinp
\end{equation*}
The set \( \closedconvexhull\np{-\theta\np{\Uncertain}} \) is closed convex
by definition, and is weakly bounded as the finite union of 
weakly bounded sets by~\eqref{eq:barrier_cone_of_a_finite_union}. 

This ends the proof.
\end{proof}

\section{Lower bound convex programs for exact sparse optimization}
\label{Dual_problems_for_exact_sparse_optimization_problems}

In this section, we consider minimization problems under the constraint that
the \pseudonormlzero\ is less than a given integer.
In~\S\ref{Constant_along_primal_rays_coupling_and_conjugacy},
we introduce and recall the main properties of the so-called
coupling \Capra~\cite{Chancelier-DeLara:2019b}.
Then, in~\S\ref{Lower_bounds_for_exact_sparse_optimization_problems},
we show how to obtain lower bounds for exact sparse optimization problems.

In this section, we work on the Euclidian space~$\RR^d$
(with $d \in \NN^*$), equipped with the scalar product 
\( \proscal{}{} \)
and with the Euclidian norm
\( \norm{\cdot} = \sqrt{ \proscal{\cdot}{\cdot} } \).

\subsection{Constant along primal rays coupling (\Capra)
and conjugacy}
\label{Constant_along_primal_rays_coupling_and_conjugacy}

To provide lower bounds, we make use of a suitable conjugacy
(not the Fenchel one) induced by a novel coupling \Capra.
This coupling has the property of being constant along primal rays, 
like the \pseudonormlzero.
The material here is mostly taken from 
the companion paper~\cite{Chancelier-DeLara:2019b}.

\subsubsection{Constant along primal rays coupling and conjugacy}

In~\cite{Chancelier-DeLara:2019b}, 
we have introduced and studied a novel coupling, 
defined as follows on the Euclidian space~$\RR^d$.

\begin{definition}
  We define the \emph{(Euclidian) coupling \Capra}~$\couplingCAPRA$ 
between $\RR^d$ and $\RR^d$ by
\begin{equation}
\forall \dual \in \RR^d \eqsepv 
\begin{cases}
  \couplingCAPRA\np{\primal, \dual} 
&= \displaystyle
\frac{ \proscal{\primal}{\dual} }{ \norm{\primal} }
\eqsepv \forall \primal \in \RR^d\backslash\{0\} \eqfinv
\\[4mm]
\couplingCAPRA\np{0, \dual} &= 0.
\end{cases}
\label{eq:coupling_CAPRAC_original}
\end{equation}
\end{definition}

The coupling \Capra\ has the property of being 
\emph{constant along primal rays}, hence the acronym~\Capra.
We introduce the \emph{Euclidian unit sphere}
\begin{equation}
  \SPHERE= \defset{\primal \in \RR^d}{\norm{\primal} = 1} 
\eqfinv 
\label{eq:Euclidian_SPHERE}
\end{equation}
and the \emph{normalization mapping}~$\normalized$
\begin{equation*}
\normalized : \RR^d \to \SPHERE \cup \{0\} 
\eqsepv
\normalized\np{\primal}=
\begin{cases}
\frac{ \primal }{ \triplenorm{\primal} }
& \mtext{ if } \primal \neq 0 \eqfinv 
\\
0 
& \mtext{ if } \primal = 0 \eqfinp
\end{cases}  
\end{equation*}

With these notations, the  \Capra\ 
coupling~\eqref{eq:coupling_CAPRAC_original} is a special case
of one-sided linear coupling $\coupling_{\normalized}$, 
as in~\eqref{eq:one-sided_linear_coupling} with \(\theta=\normalized\),
the Fenchel coupling after primal normalization:
\begin{equation*}
\couplingCAPRA\np{\primal, \dual} = 
\coupling_{\normalized}\np{\primal, \dual} = 
\proscal{\normalized\np{\primal}}{\dual} 
\eqsepv \forall \primal \in \RR^d
\eqsepv \forall \dual \in \RR^d 
\eqfinp
\label{eq:coupling_CAPRAC}
\end{equation*}

Here are expressions for the \Capra\ conjugates 
in function of the Fenchel conjugate.

\begin{subequations}
\begin{proposition}
  For any function \( \fonctiondual : \RR^d \to \barRR \), 
  the $\couplingCAPRA'$-Fenchel-Moreau conjugate is given by 
  \begin{equation}
\SFM{\fonctiondual}{\couplingCAPRA'}=
 \LFM{ \fonctiondual } \circ \normalized
\eqfinp 
\label{eq:CAPRA'_Fenchel-Moreau_conjugate}
\end{equation}
  For any function \( \fonctionprimal : \RR^d \to \barRR \), 
  the $\couplingCAPRA$-Fenchel-Moreau conjugate is given by 
   \begin{equation}
\SFM{\fonctionprimal}{\couplingCAPRA}=
 \LFM{ \bp{\InfimalPostcomposition{\normalized}{\fonctionprimal}} }
\eqfinv 
\label{eq:CAPRA_Fenchel-Moreau_conjugate}
  \end{equation}
where the infimal postcomposition~\eqref{eq:InfimalPostcomposition}
has the expression
\begin{equation}
\bp{\InfimalPostcomposition{\normalized}{\fonctionprimal}}\np{\primal}
=
\inf\defset{\fonctionprimal\np{\primal'}}{
\normalized\np{\primal'}=\primal}
=
\begin{cases}
  \inf_{\lambda > 0} \fonctionprimal\np{\lambda\primal}
& \text{if } \primal \in \SPHERE  \cup \{0\} 
\\
+\infty  
& \text{if } \primal \not\in \SPHERE \cup \{0\} 
\end{cases}
\label{eq:CAPRA_InfimalPostcomposition}
  \end{equation}
and the $\couplingCAPRA$-Fenchel-Moreau biconjugate 
is given by
\begin{equation}
  \SFMbi{\fonctionprimal}{\couplingCAPRA}
= 
\SFM{ \bp{ \SFM{\fonctionprimal}{\couplingCAPRA} } }{\star} 
\circ \normalized
=
 \LFMbi{ \bp{\InfimalPostcomposition{\normalized}{\fonctionprimal}} }
\circ \normalized
\eqfinp
\label{eq:CAPRA_Fenchel-Moreau_biconjugate}
\end{equation}
  \label{pr:CAPRA_Fenchel-Moreau_conjugate}
\end{proposition}
\end{subequations}

\subsubsection{\Capra\ conjugates and biconjugates 
related to the \pseudonormlzero}
\label{CAPRAC_conjugates_and_biconjugates_related_to_the_pseudo_norm}

Now, we will give formulas for conjugates and biconjugates 
of functions related to the \pseudonormlzero.

First, we recall the definitions of 
the $2$-$k$-symmetric gauge norm and the $k$-support norm.
For any \( \primal \in \RR^d \) and \( K \subset \ba{1,\ldots,d} \), 
we denote by
\( \primal_K \in \RR^d \) the vector which coincides with~\( \primal \),
except for the components outside of~$K$ that vanish:
\( \primal_K \) is the orthogonal projection of~\( \primal \) onto
the subspace \( \RR^K \times \{0\}^{-K} \subset \RR^d \).
Here, following notation from Game Theory, 
we have denoted by $-K$ the complementary subset 
of~$K$ in \( \ba{1,\ldots,d} \): \( K \cup (-K) = \ba{1,\ldots,d} \)
and \( K \cap (-K) = \emptyset \). 
In what follows, 
$\cardinal{K}$ denotes the cardinal of the set~$K$ and 
the notation \( \sup_{\cardinal{K} \leq k} \) 
is a shorthand for \( \sup_{ { K \subset \na{1,\ldots,d}, \cardinal{K} \leq k}} \)
(the same holds for \( \sup_{\cardinal{K} = k} \)).

\begin{definition}
  Let \( \primal \in \RR^d \).
  For \( k \in \ba{1,\ldots,d} \), 
we denote by \( \SymmetricGaugeNorm{k}{\primal} \)
  the maximum of \( \norm{\primal_K} \) over all subsets \( K \subset
  \ba{1,\ldots,d} \) with cardinal (less than)~$k$:
  \begin{equation}
\SymmetricGaugeNorm{k}{\primal}
=
\sup_{\cardinal{K} \leq k}\norm{\primal_K} 
=
\sup_{\cardinal{K} = k}\norm{\primal_K} 
    \eqfinp
    \label{eq:symmetric_gauge_norm}
  \end{equation}
  Thus defined, \( \SymmetricGaugeNorm{k}{\cdot} \) is a norm,
the \emph{$2$-$k$-symmetric gauge norm}~\cite{Mirsky:1960}.
Its dual norm (see Proposition~\ref{pr:DUAL_NORM_Euclidian}) is called
\emph{$k$-support norm~\cite{Argyriou-Foygel-Srebro:2012}}, denoted by
\( \SupportNorm{k}{\cdot} \):
\begin{equation}
  \SupportNorm{k}{\cdot} = \bp{ \SymmetricGaugeNorm{k}{\cdot} }_{\star}
\eqfinp
    \label{eq:support_norm}
\end{equation}
    \label{de:symmetric_gauge_norm}
\end{definition}
\medskip

Second, we recall the definition of the \pseudonormlzero.
We define the \emph{support} of a vector in~\( \RR^d \) by
\begin{equation*}
\Support{\primal} = 
\defset{ j \in \ba{1,\ldots,d} }{\primal_j > 0 }
\eqsepv \forall \primal \in \RR^d 
\eqfinp
\end{equation*}
The so-called \emph{\pseudonormlzero} is the function
\( \lzero : \RR^d \to \ba{0,1,\ldots,d} \)
defined by 
\begin{equation}
\lzero\np{\primal} = 
\cardinal{\Support{\primal}}
\eqsepv \forall \primal \in \RR^d 
\eqfinp
\label{eq:pseudo_norm_l0}  
\end{equation}

The \pseudonormlzero\ is used in exact sparse optimization problems of the form
\( \inf_{ \lzero\np{\primal} \leq k } \fonctionprimal\np{\primal} \).
This is why we introduce the \emph{level sets}
of the \pseudonormlzero:
\begin{equation}
  \LevelSet{\lzero}{k} 
= 
\defset{ \primal \in \RR^d }{ \lzero\np{\primal} \leq k }
\eqsepv \forall k \in \ba{0,1,\ldots,d} 
\eqfinp
\label{eq:level_set_pseudonormlzero}
  \end{equation}
\medskip

Third, we present the main result of~\cite{Chancelier-DeLara:2019b}.
  The \pseudonormlzero\ in~\eqref{eq:pseudo_norm_l0}, 
the characteristic function \( \delta_{ \LevelSet{\lzero}{k} } \) 
of its level set
and the symmetric gauge norms in~\eqref{eq:symmetric_gauge_norm}
are related by the following conjugate formulas.

\begin{theorem}
Let \( k \in \ba{0,1,\ldots,d} \). We have that:
  \begin{subequations}
    \begin{align}
      \SFM{ \delta_{ \LevelSet{\lzero}{k} } }{-\couplingCAPRA}
= \SFM{ \delta_{ \LevelSet{\lzero}{k} } }{\couplingCAPRA}
      &=
        \SymmetricGaugeNorm{k}{\cdot} 
\eqfinv
        \label{eq:conjugate_delta_l0norm}
      \\ 
      \SFMbi{ \delta_{ \LevelSet{\lzero}{k} } }{\couplingCAPRA} 
      &=
        \delta_{ \LevelSet{\lzero}{k} } 
\eqfinv
        \label{eq:biconjugate_delta_l0norm}
      \\
      \SFM{ \lzero }{\couplingCAPRA} 
      &=
\sup_{l=0,1,\ldots,d} \Bc{ \SymmetricGaugeNorm{l}{\cdot} -l }  
\eqfinv
        \label{eq:conjugate_l0norm}
      \\
      \SFMbi{ \lzero }{\couplingCAPRA} 
      &=\lzero 
        \label{eq:biconjugate_l0norm}
        \eqfinv
    \end{align}
  \end{subequations}
with the convention, in~\eqref{eq:conjugate_delta_l0norm}
and in~\eqref{eq:conjugate_l0norm}, 
that \( \SymmetricGaugeNorm{0}{\cdot} = 0 \).
\end{theorem}

\subsection{Lower bound convex program for exact sparse optimization}
\label{Lower_bounds_for_exact_sparse_optimization_problems}

With the \Capra-conjugacy recalled
in~\S\ref{Constant_along_primal_rays_coupling_and_conjugacy}, 
we now show how to obtain lower bounds 
for exact sparse optimization problems,
that are concave maximization programs.
In addition, under a mild assumption, we will show that this lower bound 
coincides with a convex minimization program
on the unit ball of the $k$-support norm, 
recalled in Definition~\ref{de:symmetric_gauge_norm}.

\begin{subequations}
\begin{theorem}
Let \( k \in \ba{0,1,\ldots,d} \). 
For any function \( \fonctionprimal : \RR^d \to \barRR \),
we have the following lower bound 
\begin{equation}
   \sup_{\dual \in \RR^d} 
    \Bp{ -\LFM{ 
\bp{\InfimalPostcomposition{\normalized}{\fonctionprimal}} 
 }\np{\dual}
      - \SymmetricGaugeNorm{k}{\dual} }
\leq 
\inf_{ \lzero\np{\primal} \leq k } \fonctionprimal\np{\primal} 
  \eqfinv
  \label{eq:duality_sparse_concave}
\end{equation}
where the dual problem to the left consists in the maximization of 
a usc concave function.

If, in addition, the convex lsc function
\( \LFM{ \bp{\InfimalPostcomposition{\normalized}{\fonctionprimal}} } \)
is proper, the above lower bound has the alternative primal expression
\begin{equation}
  \min_{ \SupportNorm{k}{\primal} \leq 1} 
  \LFMbi{ \bp{\InfimalPostcomposition{\normalized}{\fonctionprimal}} } 
  \np{\primal} 
  =
  \sup_{\dual \in \RR^d} 
  \Bp{ -\LFM{ 
      \bp{\InfimalPostcomposition{\normalized}{\fonctionprimal}} 
    }\np{\dual}
    - \SymmetricGaugeNorm{k}{\dual} }
  \leq 
  \inf_{ \lzero\np{\primal} \leq k } \fonctionprimal\np{\primal} 
  \eqfinv
  \label{eq:duality_sparse_convex}
\end{equation}
where the primal problem to the left consists in the minimization of 
a lsc convex function on the unit ball of the $k$-support norm. 
\label{th:duality_sparse_convex}
\end{theorem}
\end{subequations}

\begin{proof}
From the Inequality~\eqref{eq:dual_problem_inequality_constraints},
where we use the expression~\eqref{eq:CAPRA_Fenchel-Moreau_conjugate} 
for \( \SFM{\fonctionprimal}{\couplingCAPRA} \)
and the expression~\eqref{eq:conjugate_delta_l0norm}
for \( \SFM{ \delta_{ \LevelSet{\lzero}{k} } }{-\couplingCAPRA} \),
we deduce Inequality~\eqref{eq:duality_sparse_concave}.
Because the norm~\( \SymmetricGaugeNorm{k}{\cdot} \) 
is convex lsc and has full effective domain~$\RR^d$,
and because the convex lsc function
\( \LFM{ \bp{\InfimalPostcomposition{\normalized}{\fonctionprimal}} } \)
is proper, we deduce that 
\( \continuitypoints\np{\SymmetricGaugeNorm{k}{\cdot}} \cap 
\dom\bp{\LFM{ 
      \np{\InfimalPostcomposition{\normalized}{\fonctionprimal}}}}
= 
\dom\bp{\LFM{ 
      \np{\InfimalPostcomposition{\normalized}{\fonctionprimal}}}} 
\not= \emptyset \). 
Thus, the conditions for a Fenchel-Rockafellar equality are satisfied 
\cite[Prop.~15.13]{Bauschke-Combettes:2017} 
and we obtain that 
\begin{equation*}
  \min_{ \SupportNorm{k}{\primal} \leq 1} 
  \LFMbi{ \bp{\InfimalPostcomposition{\normalized}{\fonctionprimal}} } 
  \np{\primal} 
  = \sup_{\dual \in \RR^d} 
  \Bp{ -\LFM{ 
      \bp{\InfimalPostcomposition{\normalized}{\fonctionprimal}} 
    }\np{\dual}
    - \SymmetricGaugeNorm{k}{\dual} }
\eqfinp
\end{equation*}
This equation, combined with Equation~\eqref{eq:duality_sparse_concave}, 
gives Equation~\eqref{eq:duality_sparse_convex}.
This ends the proof.
\end{proof}


As an application, we consider 
the least squares regression sparse optimization problem.

\begin{proposition}
  Letting $\matrice$ be a matrix with $d$~rows and $p$ columns,
  and $z \in \RR^p$, we have
  \begin{multline}
      \norm{z}^2 +
    \sup_{\dual \in \RR^d}
    \Bp{
      -\Bc{
        \sup_{\primal \in \SPHERE} 
        \bp{ \proscal{\primal}{\dual}
          + \frac{ \proscal{z}{\matrice\primal}^2 }{ \norm{\matrice\primal}^2 } 
          {\mathbb I}_{\proscal{z}{\matrice\primal} > 0}}}_{+} 
      - \SymmetricGaugeNorm{k}{\dual} }
\\
=    \norm{z}^2 +
    \min_{ \SupportNorm{k}{\primal} \leq 1} 
    \LFMbi{ \Bp{ 
        -\frac{ \proscal{z}{\matrice~\cdot}^2 }{ \norm{\matrice\cdot}^2 }
        {\mathbb I}_{\proscal{z}{\matrice~\cdot} > 0}
        \UppPlus \delta_{\SPHERE} } }\np{\primal}
\leq 
    \inf_{ \lzero\np{\primal} \leq k }  \norm{z-\matrice\primal}^2 
    \eqfinp 
\label{eq:least_squares_regression_sparse_optimization_problem}
   \end{multline}
\end{proposition}

\begin{subequations}
\begin{proof}
Let $\fonctionprimal$ be the function defined by 
\( \fonctionprimal\np{\primal} = \norm{z-\matrice\primal}^2 \),
for all \( \primal\in\RR^d \). 
A straightforward calculation gives 
  \begin{align}
    \inf_{\lambda > 0} \fonctionprimal\np{\lambda\primal}
    &= 
      \norm{z}^2 
      - \frac{ \proscal{z}{\matrice\primal}^2 }{ \norm{\matrice\primal}^2 }
      {\mathbb I}_{\proscal{z}{\matrice\primal} > 0} 
\eqsepv \forall \primal\in\RR^d 
\eqfinp
  \end{align}
Therefore, using~\eqref{eq:CAPRA_InfimalPostcomposition}, 
we obtain that, for all \( \dual\in\RR^d \), 
\begin{equation}
  \LFM{\bp{\InfimalPostcomposition{\normalized}{\fonctionprimal}}}\np{\dual}
  = 
    \sup_{\primal \in \SPHERE \cup \{0\} } \Bp{ \proscal{\primal}{\dual}
      -\inf_{\lambda > 0} \fonctionprimal\np{\lambda\primal} }
=\Bc{
    \sup_{\primal \in \SPHERE} 
    \bp{ \proscal{\primal}{\dual}
      + \frac{ \proscal{z}{\matrice\primal}^2 }{ \norm{\matrice\primal}^2 } {\mathbb I}_{\proscal{z}{\matrice\primal} > 0}}}_{+} 
  - \norm{z}^2
  \eqfinp 
    \label{eq:nfA}
\end{equation}
Then, inserting the expression~\eqref{eq:nfA} 
of $\LFM{\bp{\InfimalPostcomposition{\normalized}{\fonctionprimal}}}$ 
in Inequality~\eqref{eq:duality_sparse_concave}
yields the first part of
Equation~\eqref{eq:least_squares_regression_sparse_optimization_problem},
namely
  \begin{equation*}
    \norm{z}^2 +
    \sup_{\dual \in \RR^d}
    \Bp{
      -\Bc{
        \sup_{\primal \in \SPHERE} 
        \bp{ \proscal{\primal}{\dual}
          + \frac{ \proscal{z}{\matrice\primal}^2 }{ \norm{\matrice\primal}^2 } 
          {\mathbb I}_{\proscal{z}{\matrice\primal} > 0}}}_{+} 
      - \SymmetricGaugeNorm{k}{\dual} }
    \leq
    \inf_{ \lzero\np{\primal} \leq k }  \norm{z-\matrice\primal}^2 
    \label{eq:lbA}
    \eqfinp 
  \end{equation*}
Now, since the function
$\LFM{\bp{\InfimalPostcomposition{\normalized}{\fonctionprimal}}}$ is 
easily seen to be proper by~\eqref{eq:nfA},
we can use Theorem~\ref{th:duality_sparse_convex}
and Equation~\eqref{eq:duality_sparse_convex} gives
\begin{align*}
  \min_{ \SupportNorm{k}{\primal} \leq 1} 
  \LFMr{ \bgp{
  \Bc{
  \sup_{\primal \in \SPHERE} 
  \bp{ \proscal{\primal}{\dual}
  + \frac{ \proscal{z}{\matrice\primal}^2 }{ \norm{\matrice\primal}^2 } {\mathbb
  I}_{\proscal{z}{\matrice\primal} > 0}} }_{+} }
  } - \norm{z}^2 \tag{by~\eqref{eq:nfA}}
  &=
    \min_{ \SupportNorm{k}{\primal} \leq 1} 
    \LFMbi{ \bp{\InfimalPostcomposition{\normalized}{\fonctionprimal}} } 
    \np{\primal}  \\
  &=
  \sup_{\dual \in \RR^d} 
  \Bp{ -\LFM{ 
      \bp{\InfimalPostcomposition{\normalized}{\fonctionprimal}} 
    }\np{\dual}
    - \SymmetricGaugeNorm{k}{\dual} } \\
  &\leq 
  \inf_{ \lzero\np{\primal} \leq k } \fonctionprimal\np{\primal} \tag{by~\eqref{eq:duality_sparse_convex}}
  \eqfinv
\end{align*}
which is the second part of 
Equation~\eqref{eq:least_squares_regression_sparse_optimization_problem}.
This ends the proof.
\end{proof}
\end{subequations}

\section{Lower bound convex programs for generalized sparse optimization}
\label{One-sided_convex_couplings_and_generalized_sparse_optimization}

In~\S\ref{Lower_bound_convex_programs_for_GSO}
we formally define 
generalized sparse optimization (GSO)
as optimization over the union of a finite family of subsets.
Then, we provide lower bound convex minimization programs
over the unit ball of some norms.
In~\S\ref{Design_of_norms_for_GSO}, 
we present a systematic way to design such norms,
with results on how to build up a (global) norm 
from (local) sets or norms.
Finally, in~\S\ref{Design_of_lower_bound_convex_programs_for_GSO},
we wrap up the results from the both previous subsections and
provide our main result.
Thus doing, we hint at how we encompass
most of the sparsity inducing norms used in machine learning.

\subsection{Lower bound convex programs for generalized sparse optimization}
\label{Lower_bound_convex_programs_for_GSO}

\renewcommand{\UNCERTAIN}{\mathbb{W}}
\renewcommand{\Uncertain}{W}
\renewcommand{\uncertain}{w}
\renewcommand{\PRIMAL}{\HILBERT}
\renewcommand{\DUAL}{\HILBERT}
\renewcommand{\Primal}{\Hilbert}
\renewcommand{\Dual}{\Hilbert'}
\renewcommand{\primal}{\hilbert}
\renewcommand{\dual}{\hilbert'}
\renewcommand{\PRIMAL}{{\mathbb X}}
\renewcommand{\Primal}{X}
\renewcommand{\primal}{x}
\renewcommand{\DUAL}{{\mathbb Y}}
\renewcommand{\Dual}{Y}
\renewcommand{\dual}{y}

Let $\UNCERTAIN$ be a set, let $\SCENARIO$ be a finite set and let
\(\sequence{\Uncertain_{\scenario}}{\scenario\in\SCENARIO}\)
be a family of subsets of~$\UNCERTAIN$.
This family captures \emph{sparsity}, 
where the finite set~$\SCENARIO$ of indices 
reflects the combinatorial nature of the optimization problem.

For any function \( \fonctionuncertain : \UNCERTAIN \to \barRR \),
the \emph{generalized sparse optimization (GSO)} problem is\footnote{%
The function \( \fonctionuncertain : \UNCERTAIN \to \barRR \)
needs only be known on \( \dom\fonctionuncertain \cap 
\Bp{\bigcup_{\scenario\in\SCENARIO} \Uncertain_{\scenario}} \).}
\begin{equation}
  \inf_{\uncertain \in \bigcup_{\scenario\in\SCENARIO} \Uncertain_{\scenario}} 
\fonctionuncertain\np{\uncertain} 
\eqfinp
\label{eq:GSO_problem}
\end{equation}

As the problem~\eqref{eq:GSO_problem} is a special case
of~\eqref{eq:optimization} --- with constraint given by the belonging 
of possible solutions to the set
\( \Uncertain =
\bigcup_{\scenario\in\SCENARIO} \Uncertain_{\scenario} \) ---
the following Proposition is a straightforward application of
Corollary~\ref{cor:lower_bound_concave_program_convex_program}.

\begin{proposition}
Let $\PRIMAL=\DUAL$ be a Hilbert space.
Let $\SCENARIO$ be a finite set.
Let $\UNCERTAIN$ be a set.
\begin{enumerate}
\item 
Let \(\sequence{\Uncertain_{\scenario}}{\scenario\in\SCENARIO}\)
be a family of subsets of~$\UNCERTAIN$.
\item 
Let \( \sequence{\theta_{\scenario}}{\scenario\in\SCENARIO} \)
be a family of mappings
\( \theta_{\scenario} : \Uncertain_{\scenario} \to \PRIMAL \)
such that 
\begin{enumerate}
  \item 
the family \( \sequence{\theta_{\scenario}}{\scenario\in\SCENARIO} \)
is compatible with the family 
\(\sequence{\Uncertain_{\scenario}}{\scenario\in\SCENARIO}\),
in the sense that\\
\( 
\uncertain \in \Uncertain_{\scenario} \cap \Uncertain_{\scenariobis}
\Rightarrow 
\theta_{\scenario}\np{\uncertain} 
= \theta_{\scenariobis}\np{\uncertain} 
\eqsepv \forall \np{\scenario,\scenariobis} \in\SCENARIO^2 
\), 
\label{it:convex_lower_bound_GSO_compatible_families}
\item 
every set \( -\theta_{\scenario}\np{\Uncertain_{\scenario}} \) is weakly bounded,
for every \( \scenario\in\SCENARIO \). 
\label{it:convex_lower_bound_GSO_weakly_bounded}
\end{enumerate}
\item 
Let \( \fonctionuncertain : \UNCERTAIN \to \barRR \) be a function
such that 
every function 
$\LFM{ \bp{ \InfimalPostcomposition{\theta_{\scenario}}{\fonctionuncertain}}}$ 
is proper, for every \( \scenario\in\SCENARIO \), 
and 
  \( \bigcap_{\scenario\in\SCENARIO} \dom 
  \LFM{ \bp{ \InfimalPostcomposition{\theta_{\scenario}}{\fonctionuncertain} } } 
  \not = \emptyset \).
\label{it:convex_lower_bound_GSO_function}
\end{enumerate}
  Then, we have the lower bound 
  \begin{equation}
    \min_{ \primal \in 
      \closedconvexhull\bp{ -\bigcup_{\scenario\in\SCENARIO} 
        \theta_{\scenario}\np{\Uncertain_{\scenario}} } }
    \LFM{ \Bp{ \sup_{\scenario\in\SCENARIO}
        \LFM{ \bp{ \InfimalPostcomposition{\theta_{\scenario}}{\fonctionuncertain} } } } }
    \np{\primal} 
    \leq 
    \inf_{\uncertain \in \bigcup_{\scenario\in\SCENARIO} \Uncertain_{\scenario}} 
    \fonctionuncertain\np{\uncertain}
    \eqfinp 
\label{eq:convex_lower_bound_GSO}
  \end{equation}
  \label{pr:convex_lower_bound_GSO}
\end{proposition}

\begin{proof}
  By item~\ref{it:convex_lower_bound_GSO_compatible_families},
we can define the mapping 
  \begin{equation*} 
    \theta : 
 \bigcup_{\scenario\in\SCENARIO} \Uncertain_{\scenario} \to \PRIMAL 
    \mtext{ by }
    \uncertain \in \Uncertain_{\scenario} \Rightarrow 
    \theta\np{\uncertain}=\theta_{\scenario}\np{\uncertain}
    \eqfinp
  \end{equation*} 

  By item~\ref{it:convex_lower_bound_GSO_weakly_bounded},
as every set \( -\theta_{\scenario}\np{\Uncertain_{\scenario}} \) 
is weakly bounded, for every \( \scenario\in\SCENARIO \),
 the finite union 
\( \bigcup_{\scenario\in\SCENARIO} -\theta_{\scenario}\np{\Uncertain_{\scenario}} 
= 
-\theta\bp{\bigcup_{\scenario\in\SCENARIO} \Uncertain_{\scenario} } \) 
is weakly bounded,
by item~\ref{it:barrier_cone_of_a_finite_union}
in Proposition~\ref{pr:barrier_cone}.
\medskip

From the definition~\eqref{eq:InfimalPostcomposition} 
of the infimal postcomposition, we get that 
\begin{equation*}
\bp{\InfimalPostcomposition{\theta}{\fonctionuncertain}}\np{\primal} 
  =  
\inf \defset{ \fonctionuncertain\np{\uncertain'} }%
      { \uncertain' \in \UNCERTAIN, \, 
\exists \scenario\in\SCENARIO, \, 
      \theta_{\scenario}\np{\uncertain'}=\primal }
= 
\inf_{\scenario\in\SCENARIO}
\bp{\InfimalPostcomposition{\theta_{\scenario}}{\fonctionuncertain}}\np{\primal} 
\eqfinp 
\end{equation*}
Therefore, $\LFM{\bp{\InfimalPostcomposition{\theta}{\fonctionuncertain}}} = \sup_{\scenario\in\SCENARIO} 
  \LFM{\bp{\InfimalPostcomposition{\theta_{\scenario}}{\fonctionuncertain}}}$,
as conjugacies, being dualities, turn infima into suprema.
  By item~\ref{it:convex_lower_bound_GSO_function},
the mapping $\LFM{\bp{\InfimalPostcomposition{\theta}{\fonctionuncertain}}}$ is
  proper.

To conclude, we apply Corollary~\ref{cor:lower_bound_concave_program_convex_program},
with \( \Uncertain =\bigcup_{\scenario\in\SCENARIO} \Uncertain_{\scenario} \)
and \( \closedconvexhull\bp{-\theta\np{\Uncertain}} =
\closedconvexhull\Bp{\bigcup_{\scenario\in\SCENARIO} -\theta\np{\Uncertain_{\scenario}}}
= \closedconvexhull\bp{ -\bigcup_{\scenario\in\SCENARIO} 
        \theta_{\scenario}\np{\Uncertain_{\scenario}} } \). 
\end{proof}

Going on, we provide conditions under which 
the lower bound~\eqref{eq:convex_lower_bound_GSO}
is a convex minimization program over the unit ball
of a norm (that will be detailed in~\S\ref{Design_of_norms_for_GSO}).

\begin{proposition}
Let $\PRIMAL=\DUAL$ be a Hilbert space.
Let $\SCENARIO$ be a finite set.
Let $\UNCERTAIN$ be a set.

\begin{enumerate}
\item 
Let \(\sequence{\Uncertain_{\scenario}}{\scenario\in\SCENARIO}\)
be a family of two by two disjoint subsets of~$\UNCERTAIN$.
\item 
Let \( \sequence{\theta_{\scenario}}{\scenario\in\SCENARIO} \)
be a family of mappings
\( \theta_{\scenario} : \Uncertain_{\scenario} \to \PRIMAL \),
such that 
\begin{enumerate}
\item 
the following joint full sum 
condition is satisfied
\begin{equation}
\sum_{\scenario\in\SCENARIO} 
\linearspan\bp{\theta_{\scenario}\np{\Uncertain_{\scenario}}}
= \PRIMAL 
 \eqfinv
\end{equation}
\label{subit:convex_lower_bound_GSO_orthogonality_condition_SousEnsemble}
\item 
every subset
\( \theta_{\scenario}\np{\Uncertain_{\scenario}} \) of~$\PRIMAL$
is symmetric and weakly bounded, for every \( \scenario\in\SCENARIO \).
\label{subit:convex_lower_bound_GSO_symmetric_and_weakly_bounded}
\end{enumerate}
\item 
Let \( \fonctionuncertain : \UNCERTAIN \to \barRR \) be a function
such that 
every function 
$\LFM{ \bp{ \InfimalPostcomposition{\theta_{\scenario}}{\fonctionuncertain}}}$ 
is proper, for every \( \scenario\in\SCENARIO \), and
  \( \bigcap_{\scenario\in\SCENARIO} \dom 
  \LFM{ \bp{ \InfimalPostcomposition{\theta_{\scenario}}{\fonctionuncertain} } } 
  \not = \emptyset \).
\label{it:convex_lower_bound_GSO_norm_function}
\end{enumerate}
Then, there exists a norm~\(\triplenorm{\cdot}\),
with unit ball
\( \closedconvexhull\bp{ -\bigcup_{\scenario\in\SCENARIO} 
        \theta_{\scenario}\np{\Uncertain_{\scenario}} } \),
such that we have the lower bound 
\begin{equation}
      \min_{ \triplenorm{\primal} \leq 1 }
\LFM{ \Bp{ \sup_{\scenario\in\SCENARIO}
\LFM{ \bp{ \InfimalPostcomposition{\theta_{\scenario}}{\fonctionuncertain} } } } }
\np{\primal} 
    \leq 
\inf_{\uncertain \in \bigcup_{\scenario\in\SCENARIO} \Uncertain_{\scenario}} 
\fonctionuncertain\np{\uncertain} 
\eqfinp
\end{equation}
\label{pr:convex_lower_bound_GSO_norm}
\end{proposition}

\begin{proof}
First, we use Proposition~\ref{pr:convex_lower_bound_GSO}
to obtain the lower bound~\eqref{eq:convex_lower_bound_GSO}.
For this purpose, we check its three assumptions 
(item~\ref{it:convex_lower_bound_GSO_compatible_families},
item~\ref{it:convex_lower_bound_GSO_weakly_bounded},
item~\ref{it:convex_lower_bound_GSO_function}) one by one.
\begin{itemize}
\item 
 Because the family 
\(\sequence{\Uncertain_{\scenario}}{\scenario\in\SCENARIO}\)
is made of two by two disjoint subsets of~$\UNCERTAIN$,
the compatibility condition of
item~\ref{it:convex_lower_bound_GSO_compatible_families},
in Proposition~\ref{pr:convex_lower_bound_GSO}
is satisfied. 
\item 
As every subset
\( \theta_{\scenario}\np{\Uncertain_{\scenario}} \) of~$\PRIMAL$
is symmetric and weakly bounded for every \( \scenario\in\SCENARIO \),
by item~\ref{subit:convex_lower_bound_GSO_symmetric_and_weakly_bounded}
here, we deduce that 
item~\ref{it:convex_lower_bound_GSO_weakly_bounded}
of Proposition~\ref{pr:convex_lower_bound_GSO} is satisfied.
\item 
Item~\ref{it:convex_lower_bound_GSO_function}
of Proposition~\ref{pr:convex_lower_bound_GSO}
coincides with item~\ref{it:convex_lower_bound_GSO_norm_function} here.
\end{itemize}
Therefore, we obtain the lower bound~\eqref{eq:convex_lower_bound_GSO}.
\medskip

Second, there remains to prove that there exists a norm~\(\triplenorm{\cdot}\)
with unit ball
\( \closedconvexhull\bp{ -\bigcup_{\scenario\in\SCENARIO} 
        \theta_{\scenario}\np{\Uncertain_{\scenario}} } \).
Now, this is a straightforward application of
Theorem~\ref{pr:triplenorm_local_sets} below, with 
\( \Hilbert_{\scenario}=\theta_{\scenario}\np{\Uncertain_{\scenario}} \),
for every \( \scenario\in\SCENARIO \),
and by
item~\ref{subit:convex_lower_bound_GSO_orthogonality_condition_SousEnsemble}
of the second assumption of this Proposition.

This ends the proof.
\end{proof}

\subsection{Building up a (global) norm from (local) norms}
\label{Design_of_norms_for_GSO}

\renewcommand{\UNCERTAIN}{\HILBERT}
\renewcommand{\Uncertain}{\Hilbert}
\renewcommand{\uncertain}{\hilbert}
\renewcommand{\PRIMAL}{\UNCERTAIN}
\renewcommand{\DUAL}{\UNCERTAIN}
\renewcommand{\Primal}{\Uncertain}
\renewcommand{\Dual}{\Uncertain'}
\renewcommand{\primal}{\uncertain}
\renewcommand{\dual}{\uncertain'}

Proposition~\ref{pr:convex_lower_bound_GSO_norm} claims the existence of a norm.
Here, we show how we can obtain a global norm on a Hilbert space,
first from subsets in Proposition~\ref{pr:triplenorm_local_sets}, 
second from local norms defined on closed subspaces in 
Proposition~\ref{pr:triplenorm_local_norms}.

\begin{proposition}
Let $\HILBERT$ be a Hilbert space.
Let $\SCENARIO$ be a finite set.

Let \(\sequence{\Uncertain_{\scenario}}{\scenario\in\SCENARIO}\)
be a family of subsets of~$\UNCERTAIN$
that are all symmetric and weakly bounded, that is,
\begin{equation}
-\Hilbert_{\scenario}=\Hilbert_{\scenario}
\eqsepv
\barriercone\Hilbert_{\scenario}=\HILBERT
 \eqsepv \forall \scenario\in\SCENARIO 
        \eqfinv
\label{eq:triplenorm_from_balls_local_hypothesis}
\end{equation}
and that jointly satisfy the full sum 
condition
\begin{equation}
\sum_{\scenario\in\SCENARIO} 
\linearspan\Uncertain_{\scenario}
= \HILBERT
 \eqfinp
  \label{eq:orthogonality_condition_SousEnsemble}
\end{equation}
   Then, there is a (unique) norm~\(\triplenorm{\cdot}\) on~$\PRIMAL$ 
    with unit ball \( \closedconvexhull\Bp{
        \bigcup_{\scenario\in\SCENARIO} \Hilbert_{\scenario} } \).
Moreover, the norm~\(\triplenorm{\cdot}\)
admis a dual norm~\(\triplenorm{\cdot}_\star\)
with unit ball \( \Bp{
 \bigcup_{\scenario\in\SCENARIO} \Hilbert_{\scenario} }^{\odot} \).
The norm~\(\triplenorm{\cdot}\)
and the dual norm~\(\triplenorm{\cdot}_\star\) are given by 
 \begin{subequations}
  \begin{align}
\triplenorm{\cdot} = \sigma_{ \bp{
        \bigcup_{\scenario\in\SCENARIO} \Hilbert_{\scenario} }^{\odot} }
&\mtext{ and } 
   \triplenorm{\cdot}_\star = 
\sigma_{ \closedconvexhull\bp{
        \bigcup_{\scenario\in\SCENARIO} \Hilbert_{\scenario} } }
     \eqfinv
\label{eq:triplenorm_and_triplenorm_dual} 
\intertext{ and their respective unit balls are }
\BALL_{\triplenorm{\cdot}} = \closedconvexhull\Bp{
        \bigcup_{\scenario\in\SCENARIO} \Hilbert_{\scenario} } 
&\mtext{ and }
\BALL_{\triplenorm{\cdot}_\star}= \bp{
        \bigcup_{\scenario\in\SCENARIO} \Hilbert_{\scenario} }^{\odot} 
     \eqfinp
\label{eq:triplenorm_and_triplenorm_dual_unit_balls} 
  \end{align}
\label{eq:triplenorm}
\end{subequations}
The topologies defined by the norm~\(\triplenorm{\cdot}\)
and by the dual norm~\(\triplenorm{\cdot}_\star\)
are both weaker (contain less open sets) 
than the Hilbertian topology.
  \label{pr:triplenorm_local_sets}
\end{proposition}

\begin{proof}
We prove that the closed convex set 
\begin{equation}
  \BALL = \closedconvexhull\bp{
        \bigcup_{\scenario\in\SCENARIO} \Hilbert_{\scenario} } 
      \label{eq:triplenorm_unit_ball}
\end{equation}
satisfies the following conditions of 
Proposition~\ref{pr:norm_and_dual_norm_Hilbert}, namely 
\begin{equation*}
  -\BALL=\BALL 
\eqsepv
\barriercone\BALL = \DUAL 
\eqsepv 
\conicalhull\BALL= \PRIMAL  
\eqfinp
\end{equation*}
It is clear that \( -\BALL=\BALL \) since 
\( -\Hilbert_{\scenario}=\Hilbert_{\scenario} \) for all 
\( \scenario\in\SCENARIO\)
 by~\eqref{eq:triplenorm_from_balls_local_hypothesis}.

We show that \( \barriercone\BALL = \DUAL \):
\begin{align*}
  \barriercone\BALL 
&= 
\barriercone\Bp{\closedconvexhull\bp{
        \bigcup_{\scenario\in\SCENARIO} \Hilbert_{\scenario} } }
\tag{ by~\eqref{eq:triplenorm_unit_ball} }
\\
&= 
\barriercone\bp{
        \bigcup_{\scenario\in\SCENARIO} \Hilbert_{\scenario} }
\tag{ by~\eqref{eq:barrier_cone_and_convexhull} }
\\
&=
\bigcap_{\scenario\in\SCENARIO} \barriercone\Hilbert_{\scenario}
\tag{ by~\eqref{eq:barrier_cone_of_a_finite_union} }
\\
&=
\bigcap_{\scenario\in\SCENARIO} \PRIMAL 
\tag{ by~\eqref{eq:triplenorm_from_balls_local_hypothesis} }
\\
&= \PRIMAL
\eqfinp 
\end{align*}

There remains to show that \( \conicalhull\BALL= \PRIMAL \):
\begin{align*}
   \conicalhull\BALL
&= 
\linearspan\BALL
\tag{ by~\cite[Prop.~6.4]{Bauschke-Combettes:2017} 
as $\BALL$ is nonempty convex and symmetric}
\\
&= 
\linearspan\Bp{ \closedconvexhull\bp{
        \bigcup_{\scenario\in\SCENARIO} \Hilbert_{\scenario} } }
\tag{ by~\eqref{eq:triplenorm_unit_ball} }
\\
& \supset
\linearspan\Bp{ \convexhull\bp{
        \bigcup_{\scenario\in\SCENARIO} \Hilbert_{\scenario} } }
\\
&= 
\linearspan\bp{ 
\bigcup_{\scenario\in\SCENARIO} \Hilbert_{\scenario} }
\\
&= 
\sum_{\scenario\in\SCENARIO} \linearspan\Hilbert_{\scenario}
\\
&=
\PRIMAL
\tag{ by~\eqref{eq:orthogonality_condition_SousEnsemble} }
\eqfinp
\end{align*}

We have proved that the closed convex set~$\BALL$
in~\eqref{eq:triplenorm_unit_ball} satisfies the conditions
of Proposition~\ref{pr:norm_and_dual_norm_Hilbert}.
We conclude that 
\( \sigma_{\BALL^{\odot}} \) is a norm~\(\triplenorm{\cdot}\) on~$\PRIMAL$ 
    with unit ball~$\BALL$, 
and that it admits the dual norm 
\( \triplenorm{\cdot}_{\star} = \sigma_{\BALL} \),
    with unit ball~$\BALL^{\odot}$.
This gives~\eqref{eq:triplenorm}.

In addition, the topologies defined by the norm~\(\triplenorm{\cdot}\)
and by the dual norm~\(\triplenorm{\cdot}_\star\)
are both weaker (contain less open sets) 
than the Hilbertian topology,
by Proposition~\ref{pr:comparison_of_norms} since both unit balls
$\BALL$ and $\BALL^{\odot}$ are closed by construction.

This ends the proof. 
\end{proof}

Here, we show how we can obtain a global norm on a Hilbert space,
from local norms defined on closed subspaces.
With this formulation, we are able to give expressions 
of the global norm and of its dual norm as 
convolution and supremum of local norms and dual norms.

\begin{proposition}
Let $\HILBERT$ be a Hilbert space.
Let $\SCENARIO$ be a finite set.
\begin{itemize}
\item 
Let \(\sequence{\Subspace_{\scenario}}{\scenario\in\SCENARIO}\)
be a family of closed subspaces 
of the Hilbert space~$\PRIMAL$, will full sum, that is, 
such that 
  \begin{equation}
\sum_{\scenario\in\SCENARIO}
\Subspace_{\scenario} = \PRIMAL
    \eqfinp
    \label{eq:orthogonality_condition}
  \end{equation}
\item 
Let \( \sequence{\triplenorm{\cdot}_{\scenario}}{\scenario\in\SCENARIO} \)
be a family of (local) norms on the closed
subspaces~\(\sequence{\Subspace_{\scenario}}{\scenario\in\SCENARIO}\),
such that, for every \( \scenario\in\SCENARIO \),
the norm~$\triplenorm{\cdot}_{\scenario}$ 
is equivalent to the restriction to~$\Subspace_{\scenario}$
of the Hilbertian norm~\( \norm{\cdot} \) on~$\PRIMAL$.
We define, for every \( \scenario\in\SCENARIO \),
the (local) unit ball 
\begin{equation}
        \BALL_{\scenario} 
 =
        \defset{\primal\in\Subspace_{\scenario}}%
        {\triplenorm{\primal}_{\scenario} \leq  1}
        \subset \Subspace_{\scenario}
 \eqsepv \forall \scenario\in\SCENARIO
\eqfinv
\label{eq:(local)_unit_balls}
\end{equation}
\end{itemize}
Then, there is a (unique) norm~\(\triplenorm{\cdot}\) on~$\PRIMAL$ 
    with unit ball \( \closedconvexhull\Bp{
        \bigcup_{\scenario\in\SCENARIO} \BALL_{\scenario} } \)
and it 
admis a dual norm~\(\triplenorm{\cdot}_\star\)
with unit ball \( \Bp{
 \bigcup_{\scenario\in\SCENARIO} \BALL_{\scenario} }^{\odot} \).
Moreover, the norm~\(\triplenorm{\cdot}\) 
and the dual norm~\(\triplenorm{\cdot}_\star\)
are equivalent to the Hilbertian norm~\( \norm{\cdot} \),
and have the following expressions. 
\begin{enumerate}
\item 
The norm~\(\triplenorm{\cdot}\) 
can be expressed as a convolution of the local norms
\( \sequence{\triplenorm{\cdot}_{\scenario}}{\scenario\in\SCENARIO} \):
\begin{subequations}
  \begin{align}
    \triplenorm{\cdot}  
    &= 
      \bigbox_{\scenario\in\SCENARIO} \bp{\triplenorm{\cdot}_{\scenario} 
      + \delta_{\HILBERT_\scenario}}
      \eqfinv
    \\      
    \triplenorm{\primal} 
    &=
      \inf_{ \primal^{\scenario} \in \Subspace_{\scenario}, \sum_{\scenario\in\SCENARIO}\primal^{\scenario}=\primal }
      \sum_{\scenario\in\SCENARIO} \triplenorm{ \primal^{\scenario} }_{\scenario}
      \eqsepv \forall \primal\in\PRIMAL
      \eqfinp
  \end{align}
  \label{eq:triplenorm_convolution}
\end{subequations}
\item 
For each \( \scenario\in\SCENARIO \),
the local norm~$\triplenorm{\cdot}_{\scenario}$ on~$\Subspace_{\scenario}$
admits a local dual norm~$\triplenorm{\cdot}_{\scenario,\star}$ 
on~$\Subspace_{\scenario}$, 
and the dual norm~\(\triplenorm{\cdot}_\star\), 
of the norm~\(\triplenorm{\cdot}\) 
can be expressed as a supremum of the local dual norms
\( \sequence{\triplenorm{\cdot}_{\scenario,\star}}{\scenario\in\SCENARIO} \):
  \begin{subequations}
    \begin{align}
       \triplenorm{\cdot}_\star 
      &=
      \sup_{\scenario\in\SCENARIO} 
 \bp{ \triplenorm{\cdot}_{\scenario,\star} \circ \pi_{\scenario} }
      \eqfinv
\\
\triplenorm{\dual}_\star 
&= 
     \sup_{\scenario\in\SCENARIO} 
      \triplenorm{\pi_{\scenario}\np{\dual}}_{\scenario,\star} 
\eqsepv \forall \dual \in \DUAL 
      \eqfinv
   \end{align}
      \label{eq:triplenorm_dual_supremum}
  \end{subequations}
where, for every \( \scenario\in\SCENARIO \),
we introduce the orthogonal projection mapping
onto the closed subspace~\( \Subspace_{\scenario} \)
\begin{equation}
  \pi_{\scenario}: \PRIMAL \to 
\Subspace_{\scenario}
\text{\quad such that \quad} 
\pi_{\scenario}\np{\primal} \in \Subspace_{\scenario}
\eqsepv 
\primal-\pi_{\scenario}\np{\primal} \perp
\Subspace_{\scenario}
\eqsepv 
 \forall \primal \in \PRIMAL 
\eqfinp
\label{eq:orthogonal_projection_mapping}
\end{equation}
\end{enumerate}
\label{pr:triplenorm_local_norms}
\end{proposition}

\begin{proof}
First, 
we establish two useful properties of 
the local unit balls~\( \BALL_{\scenario} \) in~\eqref{eq:(local)_unit_balls}.
By assumption, for every \( \scenario\in\SCENARIO \),
$\triplenorm{\cdot}_{\scenario}$ is a norm on~$\Subspace_{\scenario}$
which is equivalent to the restriction to~$\Subspace_{\scenario}$
of the Hilbertian norm~\( \norm{\cdot} \) on~$\PRIMAL$.
Therefore, 
for every \( \scenario\in\SCENARIO \),
every local unit ball~\( \BALL_{\scenario} \) is 
\begin{itemize}
\item 
bounded (for the Hilbertian norm~\( \norm{\cdot} \)),
by Proposition~\ref{pr:comparison_of_norms} because
there exists $m_{\scenario}>0$ such that 
\( m_{\scenario}\norm{\cdot} \leq \triplenorm{\cdot}_{\scenario} \)
on~$\Subspace_{\scenario}$,
hence weakly bounded by~\eqref{eq:barrier_cone_of_a_bounded_set},
\item 
closed in~$\Subspace_{\scenario}$
(for the relative Hilbertian topology of~$\Subspace_{\scenario}$),
by Proposition~\ref{pr:comparison_of_norms}
because there exists $M_{\scenario}>0$ such that 
\( \triplenorm{\cdot}_{\scenario} \leq M_{\scenario}\norm{\cdot} \)
on~$\Subspace_{\scenario}$.
\end{itemize}

Second,
we prove that there is a (unique) norm~\(\triplenorm{\cdot}\) on~$\PRIMAL$ 
    with unit ball \( \closedconvexhull\bp{
        \bigcup_{\scenario\in\SCENARIO} \BALL_{\scenario} } \)
and it admis a dual norm~\(\triplenorm{\cdot}_\star\)
with unit ball \( \bp{
 \bigcup_{\scenario\in\SCENARIO} \BALL_{\scenario} }^{\odot} \).
For this purpose, it suffices to show that the family
\( \sequence{\BALL_{\scenario}}{\scenario\in\SCENARIO} \)
of local unit balls, defined in~\eqref{eq:(local)_unit_balls},
satisfies the assumptions of 
Proposition~\ref{pr:triplenorm_local_sets}.

Now, for every \( \scenario\in\SCENARIO \),
every local unit ball~$\BALL_{\scenario}$ is symmetric, 
and,
by~\eqref{eq:barrier_cone_of_a_bounded_set}, is also weakly bounded since it is bounded.
Thus, the assumptions~\eqref{eq:triplenorm_from_balls_local_hypothesis}
are satisfied.
There remains to prove the full sum 
condition~\eqref{eq:orthogonality_condition_SousEnsemble}.
But it follows from an easily proven property of a unit ball ---
namely, that \( \linearspan\BALL_{\scenario}=
\closedspan\BALL_{\scenario}=\Subspace_{\scenario}\), 
for every \( \scenario\in\SCENARIO \) --- from which we get
\( \sum_{\scenario\in\SCENARIO} \linearspan\BALL_{\scenario}
= \sum_{\scenario\in\SCENARIO} \Subspace_{\scenario}
= \PRIMAL \)
by~\eqref{eq:orthogonality_condition}.

Moreover, Proposition~\ref{pr:triplenorm_local_sets} establishes that 
the topologies defined by the norm~\(\triplenorm{\cdot}\)
and by the dual norm~\(\triplenorm{\cdot}_\star\)
are both weaker (contain less open sets) 
than the Hilbertian topology.
Yet, by Proposition~\ref{pr:comparison_of_norms}, 
the topology induced by the norm~\(\triplenorm{\cdot}\) 
is stronger than the Hilbertian topology,
because the unit ball~\( \BALL_{\triplenorm{\cdot}}=\closedconvexhull\bp{
        \bigcup_{\scenario\in\SCENARIO} \BALL_{\scenario} } \) is bounded,
as finite union of bounded local unit balls.
Therefore, the norm~\(\triplenorm{\cdot}\) 
is equivalent to the Hilbertian norm~\( \norm{\cdot} \).
We conclude this part with Proposition~\ref{pr:dual_norm_Hilbert}
that asserts that the dual norm~\(\triplenorm{\cdot}_\star\)
is then also equivalent to the Hilbertian norm~\( \norm{\cdot} \).
\medskip

Third, we prove the two items (but in reverse order).
  \begin{enumerate}
  \item[2.]

By Proposition~\ref{pr:dual_norm_Hilbert}
(applied on each of the Hilbert closed subspace~$\Subspace_{\scenario}$),
for each \( \scenario\in\SCENARIO \)
the local norm~$\triplenorm{\cdot}_{\scenario}$ on~$\Subspace_{\scenario}$
admits a local dual norm~$\triplenorm{\cdot}_{\scenario,\star}$ 
on~$\Subspace_{\scenario}$. 
Indeed, we have seen 
that every local unit ball \( \BALL_{\scenario} \) is 
weakly bounded, hence weakly bounded on~$\Subspace_{\scenario}$
by~\eqref{eq:barrier_cone_of_a_bounded_set},
and closed in~$\Subspace_{\scenario}$
(for the Hilbertian relative topology).

We prove~\eqref{eq:triplenorm_dual_supremum}.
For this purpose, 
let \( { \BALL_{\scenario}^{{\odot}_{\scenario}} }= 
\defset{\dual \in \DUAL_\scenario}{\proscal{\primal}{\dual} \leq 1 
      \eqsepv \forall \primal \in \BALL_{\scenario} } \)
denote the polar of the set
$\BALL_{\scenario}$ in $\HILBERT_\scenario$, as in~\eqref{eq:polar_set}. 
We have 
\begin{subequations}
\begin{align}
\triplenorm{\primal}_{\scenario}
&=\sigma_{ \BALL_{\scenario}^{{\odot}_{\scenario}} }\np{\primal} 
\eqsepv \forall \primal \in \Subspace_{\scenario} \eqfinv
\mtext{\qquad by~\eqref{eq:norm_dual_norm=support_of_polar_balls_Hilbert}, }
\label{eq:local_norm=support}
\\
\triplenorm{\dual}_{\scenario,\star}
&=\sigma_{ \BALL_{\scenario} }\np{\dual} 
\eqsepv \forall \dual \in \Subspace_{\scenario} \eqfinv
\mtext{\qquad by~\eqref{eq:norm_dual_norm=support_of_polar_balls_Hilbert}, }
\label{eq:local_dualnorm=support}
\\
    \BALL_{\scenario}^{{\odot}} 
    &=
      \BALL_{\scenario}^{{\odot}_{\scenario}} + \HILBERT_{\scenario}^{\perp} 
      \eqfinv
\mtext{\qquad as easily deduced from the definition~\eqref{eq:polar_set} of 
a polar set }
    \\
    \sigma_{\BALL_{\scenario}^{{\odot}}}
    &=
      \sigma_{ \BALL_{\scenario}^{{\odot}_{\scenario}} } \LowPlus
      \sigma_{ \HILBERT_{\scenario}^{\perp} } 
      = \triplenorm{\cdot}_{\scenario} 
      \LowPlus \delta_{ \HILBERT_{\scenario} }
      \eqfinv
      \label{eq:local_support_polar}
\end{align}
  \end{subequations}
by~\eqref{eq:local_norm=support} and
by a property of the support function of a vector space.

For all $\dual \in \DUAL$, we have
\begin{align}
  \sigma_{ \Hilbert_{\scenario} }\np{\dual}
  &=
    \sup_{\primal \in \Hilbert_{\scenario} } 
    \bp{ \proscal{\primal}{\pi_{\scenario}\np{\dual}} + 
    \proscal{\primal}{\dual-\pi_{\scenario}\np{\dual}} }
    \tag{ where the mapping~$\pi_{\scenario}$
is defined in~\eqref{eq:orthogonal_projection_mapping} }
\nonumber
  \\
  &=\sup_{\primal \in \Hilbert_{\scenario} } 
    \proscal{\primal}{\pi_{\scenario}\np{\dual}} 
    \tag{ by property of the orthogonal projection
    mapping~\eqref{eq:orthogonal_projection_mapping} 
since \( \Hilbert_{\scenario} \subset \Subspace_{\scenario} \)}
\nonumber
  \\
  &=
\sigma_{ \Hilbert_{\scenario} }\bp{ \pi_{\scenario}\np{\dual} } 
    \eqfinp
\label{eq:support_local_ball}
\end{align}
Now, we are ready to prove~\eqref{eq:triplenorm_dual_supremum}:
\begin{align*}
  \triplenorm{\cdot}_{\star} 
&= 
\sigma_{ \closedconvexhull\bp{
\bigcup_{\scenario\in\SCENARIO} \BALL_{\scenario} } }
\tag{ by~\eqref{eq:triplenorm_and_triplenorm_dual} }
\\
&= 
\sigma_{ \bigcup_{\scenario\in\SCENARIO} \BALL_{\scenario} }
\\
&= 
\sup_{\scenario\in\SCENARIO} \sigma_{ \BALL_{\scenario} }
\\
&= 
\sup_{\scenario\in\SCENARIO} \sigma_{ \BALL_{\scenario} } \circ \pi_{\scenario}
    \tag{ by~\eqref{eq:support_local_ball} }
\\
&= 
\sup_{\scenario\in\SCENARIO} 
\bp{ \triplenorm{\cdot}_{\scenario,\star} \circ \pi_{\scenario} }
 \tag{ by~\eqref{eq:local_dualnorm=support} }
    \eqfinp
\end{align*}

\item[1.]
We prove~\eqref{eq:triplenorm_convolution}.

For this purpose, we  start by showing that 
\( 0 \in \bigcap_{\scenario\in\SCENARIO} 
\continuitypoints\bp{\delta_{ \BALL_{\scenario}^{\odot} }} \).
We have proven at the beginning
that every local unit ball~\( \BALL_{\scenario} \) 
is bounded. Therefore, \( \bigcup_{\scenario\in\SCENARIO} \BALL_{\scenario} \)
is bounded because the set~$\SCENARIO$ is finite. 
Letting $M>0$ be such that 
\( \bigcup_{\scenario\in\SCENARIO} \BALL_{\scenario} \subset M
\BALL_{\norm{\cdot}} \),
we get that 
\begin{align*}
\bigcup_{\scenario\in\SCENARIO} \BALL_{\scenario} \subset M \BALL_{\norm{\cdot}} 
&\Rightarrow 
\bp{M \BALL_{\norm{\cdot}}}^{\odot} \subset 
\Bp{ \bigcup_{\scenario\in\SCENARIO} \BALL_{\scenario} }^{\odot} 
\\
&\Rightarrow 
\frac{1}{M}  \BALL_{\norm{\cdot}_\star} \subset 
\Bp{ \bigcup_{\scenario\in\SCENARIO} \BALL_{\scenario} }^{\odot} 
\tag{ by~\eqref{eq:norm_and dual_norm_unit_ball} 
and the definition~\eqref{eq:polar_set} of a polar set }
\\
&\Rightarrow \frac{1}{M} \BALL_{\norm{\cdot}}
 \subset 
\Bp{ \bigcup_{\scenario\in\SCENARIO} \BALL_{\scenario} }^{\odot} 
\tag{ because the dual norm~$\norm{\cdot}_\star$ of the Hilbertian norm
is the Hilbertian norm}
\\
&\Rightarrow 
0 \in \interior\Bp{ \bigcup_{\scenario\in\SCENARIO} \BALL_{\scenario} }^{\odot} 
=
\interior\bigcap_{\scenario\in\SCENARIO} \BALL_{\scenario}^{\odot} 
\eqfinv
  \end{align*}
where the interior is with respect to the Hilbertian topology.
Now, as it is always true that 
\( \interior\bigcap_{\scenario\in\SCENARIO} \BALL_{\scenario}^{\odot} 
\subset
\bigcap_{\scenario\in\SCENARIO} \interior\BALL_{\scenario}^{\odot} \),
we get that 
\( 0 \in \bigcap_{\scenario\in\SCENARIO} \interior\BALL_{\scenario}^{\odot} \).
Finally, it is easily seen that 
\(\continuitypoints\bp{\delta_{ \BALL_{\scenario}^{\odot} }} =
\interior\BALL_{\scenario}^{\odot} \), 
for every \( \scenario\in\SCENARIO \).
We conclude that 
\( 0 \in \bigcap_{\scenario\in\SCENARIO} 
\continuitypoints\bp{\delta_{ \BALL_{\scenario}^{\odot} }} \).

Now, we are ready to prove~\eqref{eq:triplenorm_convolution}:
\begin{align*}
  \triplenorm{\cdot} 
  &= 
    \sigma_{\BALL^{\odot}} 
    \tag{ by~\eqref{eq:triplenorm} and the
    definition~\eqref{eq:triplenorm_unit_ball} of~$\BALL$ }
  \\
  &= 
    \LFM{ \delta_{\BALL^{\odot}} }
  \\
  &= 
    \LFM{ \delta_{ \bigcap_{\scenario\in\SCENARIO} \BALL_{\scenario}^{\odot} } }
    \tag{ as \( \BALL^{\odot} = \Bp{ \closedconvexhull\bp{
    \bigcup_{\scenario\in\SCENARIO} \BALL_{\scenario} } }^{\odot} 
    = 
    \bp{ \bigcup_{\scenario\in\SCENARIO} \BALL_{\scenario} }^{\odot} 
    = 
    \bigcap_{\scenario\in\SCENARIO} \BALL_{\scenario}^{\odot} \) }
  \\
  &= 
    \LFM{ \Bp{ \sum_{\scenario\in\SCENARIO} 
    \delta_{ \BALL_{\scenario}^{\odot} } } } 
    \tag{ as \( 
\delta_{ \bigcap_{\scenario\in\SCENARIO} \BALL_{\scenario}^{\odot} }
= \sum_{\scenario\in\SCENARIO}\delta_{ \BALL_{\scenario}^{\odot} } \) }
  \\
  &=
    \bigbox_{\scenario\in\SCENARIO} 
    \LFM{ \delta_{ \BALL_{\scenario}^{\odot} } }
    \tag{ because \( 0 \in \bigcap_{\scenario\in\SCENARIO} \continuitypoints
    \bp{\delta_{ \BALL_{\scenario}^{\odot} }} \) as shown above
    and using~\cite[Prop.~15.3-15.5]{Bauschke-Combettes:2017} }
\\
  &=
    \bigbox_{\scenario\in\SCENARIO} 
    \sigma_{ { \BALL_{\scenario}^{\odot} } }
  \\
  &=
    \bigbox_{\scenario\in\SCENARIO} 
    \bp{ \triplenorm{\cdot}_{\scenario} + \delta_{\HILBERT_\scenario}}
    \tag{ by~\eqref{eq:local_support_polar} }
    \eqfinp
\end{align*}
  \end{enumerate}

This ends the proof. 
\end{proof}

\subsection{Design of norms for lower bound convex programs for GSO}
\label{Design_of_lower_bound_convex_programs_for_GSO}

\renewcommand{\UNCERTAIN}{\HILBERT}
\renewcommand{\Uncertain}{\Hilbert}
\renewcommand{\uncertain}{\hilbert}
\renewcommand{\PRIMAL}{\UNCERTAIN}
\renewcommand{\DUAL}{\UNCERTAIN}
\renewcommand{\Primal}{\Uncertain}
\renewcommand{\Dual}{\Uncertain'}
\renewcommand{\primal}{\uncertain}
\renewcommand{\dual}{\uncertain'}

\renewcommand{\Uncertain}{W}
\renewcommand{\uncertain}{w}

Finally, we consider a generalized sparse optimization problem
on a Hilbert space
and we present a systematic way to design norms, 
that mixes the formulations and results 
of Proposition~\ref{pr:convex_lower_bound_GSO_norm} 
and Proposition~\ref{pr:triplenorm_local_norms}. 
In what follows, \emph{sparsity} is captured by a finite family
\( \sequence{\Uncertain_{\scenario}}{\scenario\in\SCENARIO} \)
of subsets of a Hilbert space~$\UNCERTAIN$, 
whereas \emph{amplitude} is measured by a family
\( \sequence{\triplenorm{\cdot}_{\scenario}}{\scenario\in\SCENARIO} \)
of (local) norms on every
closed subspace~\( \closedspan\Uncertain_{\scenario} \).

\begin{theorem}
Let $\UNCERTAIN$ be a Hilbert space.
Let $\SCENARIO$ be a finite set.
\begin{enumerate}
\item 
Let \(\sequence{\Uncertain_{\scenario}}{\scenario\in\SCENARIO}\)
be a family of two by two disjoint symmetric subsets of~$\UNCERTAIN$
such that the closed subspaces
\begin{equation}
\Subspace_{\scenario} =\closedspan\Uncertain_{\scenario} 
  \eqsepv \forall \scenario\in\SCENARIO
   \label{eq:Subspace=closedspanUncertain}
\end{equation}
generate a full sum as follows
\begin{equation}
\sum_{\scenario\in\SCENARIO} 
\Subspace_{\scenario} 
= \UNCERTAIN 
 \eqfinp
\label{eq:orthogonality_condition_SousEnsemble_bis}
\end{equation}
\label{it:main_disjoint_symmetric_subsets}
\item 
Let \( \sequence{\triplenorm{\cdot}_{\scenario}}{\scenario\in\SCENARIO} \)
be a family of (local) norms
such that, for every \( \scenario\in\SCENARIO \),
$\triplenorm{\cdot}_{\scenario}$ is a norm on
\( \Subspace_{\scenario} \), 
which is equivalent to the restriction to~$\Subspace_{\scenario}$
of the Hilbertian norm~\( \norm{\cdot} \) on~$\PRIMAL$.
We denote the 
(local) unit balls and spheres by 
  \begin{subequations}
    \begin{align}
     \BALL_{\scenario} 
      &= 
        \defset{\primal\in\Subspace_{\scenario}}%
        {\triplenorm{\primal}_{\scenario} \leq  1}
        \subset \Subspace_{\scenario}
 \eqsepv \forall \scenario\in\SCENARIO
        \eqfinv
\label{eq:local_unit_ball}
       \\
      \SPHERE_{\scenario} 
      &= 
        \defset{\primal\in\Subspace_{\scenario}}%
        {\triplenorm{\primal}_{\scenario} = 1}
        \subset \Subspace_{\scenario}
 \eqsepv \forall \scenario\in\SCENARIO
\eqfinp
 \label{eq:local_unit_sphere}
    \end{align}
  \end{subequations}
\label{it:main_(local)_norms}
\item 
Let \( \sequence{\theta_{\scenario}}{\scenario\in\SCENARIO} \)
be a family of symmetric\footnote{%
that is, 
\( \theta_{\scenario}\np{-\uncertain}=\theta_{\scenario}\np{\uncertain} \),
for all \( \uncertain\in\Uncertain_{\scenario}\) and 
for every \( \scenario\in\SCENARIO \).
}
 mappings
\( \theta_{\scenario} : \Uncertain_{\scenario} \to \PRIMAL \),
such that 
\begin{equation}
 \SPHERE_{\scenario} \subset 
\overline{ \theta_{\scenario}\np{\Uncertain_{\scenario}} }
\subset \BALL_{\scenario} 
\eqfinp
 \label{eq:cor_triplenorm_SPHERE_BALL}
\end{equation}
\label{it:main_mappings}
\item 
Let \( \fonctionuncertain : \UNCERTAIN \to \barRR \) be a function
such that every function 
$\LFM{ \bp{ \InfimalPostcomposition{\theta_{\scenario}}{\fonctionuncertain}}}$ 
is proper, for every \( \scenario\in\SCENARIO \), and
  \( \bigcap_{\scenario\in\SCENARIO} \dom 
  \LFM{ \bp{ \InfimalPostcomposition{\theta_{\scenario}}{\fonctionuncertain} } } 
  \not = \emptyset \).
\label{it:main_function}
\end{enumerate}
Then, there exists a norm~\(\triplenorm{\cdot}\),
with unit ball 
\( \closedconvexhull\Bp{
\bigcup_{\scenario\in\SCENARIO} \BALL_{\scenario} } \)
such that we have the lower bound 
\begin{equation}
      \min_{ \triplenorm{\primal} \leq 1 }
\LFM{ \Bp{ \sup_{\scenario\in\SCENARIO}
\LFM{ \bp{ \InfimalPostcomposition{\theta_{\scenario}}{\fonctionuncertain} } } } }
\np{\primal} 
    \leq 
\inf_{\uncertain \in \bigcup_{\scenario\in\SCENARIO} \Uncertain_{\scenario}} 
\fonctionuncertain\np{\uncertain} 
\eqfinp
\end{equation}
Moreover, expressions for the norm~\(\triplenorm{\cdot}\)
and for its dual norm can be found in 
Proposition~\ref{pr:triplenorm_local_norms}.
\label{th:main}
\end{theorem}

\begin{proof}
  First, we use Proposition~\ref{pr:convex_lower_bound_GSO}
to obtain the lower bound~\eqref{eq:convex_lower_bound_GSO}.
For this purpose, we check its three assumptions 
(item~\ref{it:convex_lower_bound_GSO_compatible_families},
item~\ref{it:convex_lower_bound_GSO_weakly_bounded},
item~\ref{it:convex_lower_bound_GSO_function}) one by one.
\begin{itemize}
\item 
 By item~\ref{it:main_disjoint_symmetric_subsets} here,
the family \(\sequence{\Uncertain_{\scenario}}{\scenario\in\SCENARIO}\)
is made of two by two disjoint subsets of~$\UNCERTAIN$. Therefore, 
item~\ref{it:convex_lower_bound_GSO_compatible_families}
of Proposition~\ref{pr:convex_lower_bound_GSO}
is satisfied. 
\item 
For every \( \scenario\in\SCENARIO \), every subset
\( \theta_{\scenario}\np{\Uncertain_{\scenario}} \) is symmetric, 
because so are the subsets~$\Uncertain_{\scenario}$
(by item~\ref{it:main_disjoint_symmetric_subsets} here) 
and
the mappings~\( \theta_{\scenario} \) 
(by item~\ref{it:main_mappings} here).
For every \( \scenario\in\SCENARIO \), every subset
\( -\theta_{\scenario}\np{\Uncertain_{\scenario}} 
=\theta_{\scenario}\np{\Uncertain_{\scenario}} \) is weakly bounded,
because, by~\eqref{eq:cor_triplenorm_SPHERE_BALL},
it is a subset of the 
ball~\eqref{eq:local_unit_ball}, which is bounded
(by Proposition~\ref{pr:comparison_of_norms}, 
as seen at the beginning of the proof of 
Proposition~\ref{pr:convex_lower_bound_GSO_norm}).
Therefore, item~\ref{it:convex_lower_bound_GSO_weakly_bounded}
of Proposition~\ref{pr:convex_lower_bound_GSO} is satisfied.
\item 
Item~\ref{it:convex_lower_bound_GSO_function}
of Proposition~\ref{pr:convex_lower_bound_GSO}
coincides with item~\ref{it:main_function} here.
\end{itemize}

Second, we prove that the term
\( \closedconvexhull\Bp{ -\bigcup_{\scenario\in\SCENARIO} 
\theta_{\scenario}\np{\Uncertain_{\scenario}} } 
= \closedconvexhull\Bp{ \bigcup_{\scenario\in\SCENARIO} 
\theta_{\scenario}\np{\Uncertain_{\scenario}} } \)
in the lower bound~\eqref{eq:convex_lower_bound_GSO}
can be replaced by 
\( \closedconvexhull\Bp{
\bigcup_{\scenario\in\SCENARIO} \BALL_{\scenario} } \).
Indeed 
\begin{align*}
\closedconvexhull\Bp{ \bigcup_{\scenario\in\SCENARIO} 
\BALL_{\scenario} }
&\supset 
\closedconvexhull\Bp{ \bigcup_{\scenario\in\SCENARIO} 
\overline{ \theta_{\scenario}\np{\Uncertain_{\scenario}} } }
\tag{ by \( \overline{ \theta_{\scenario}\np{\Uncertain_{\scenario}} }
\subset \BALL_{\scenario} \) 
in~\eqref{eq:cor_triplenorm_SPHERE_BALL} } 
\\
&=\closedconvexhull\Bp{ \bigcup_{\scenario\in\SCENARIO} 
\theta_{\scenario}\np{\Uncertain_{\scenario}} } 
\tag{ as easily proven}
\\
&\supset 
\closedconvexhull\Bp{
\bigcup_{\scenario\in\SCENARIO} \SPHERE_{\scenario} } 
\tag{ by \( \SPHERE_{\scenario} \subset 
\overline{ \theta_{\scenario}\np{\Uncertain_{\scenario}} } \)
in~\eqref{eq:cor_triplenorm_SPHERE_BALL} }
\\
&=  
\overline{\convexhull\Bp{
\bigcup_{\scenario\in\SCENARIO} \SPHERE_{\scenario} } }
\\
&=  
\overline{\convexhull\Bp{
\bigcup_{\scenario\in\SCENARIO} \convexhull\SPHERE_{\scenario} } }
\tag{ as easily proved }
\\
&=  
\overline{\convexhull\Bp{
\bigcup_{\scenario\in\SCENARIO} \BALL_{\scenario} } }
\tag{ because \( \convexhull\SPHERE_{\scenario}=\BALL_{\scenario} \),
\( \forall \scenario\in\SCENARIO \) 
  }
\\
&=  
\closedconvexhull\Bp{
\bigcup_{\scenario\in\SCENARIO} \BALL_{\scenario} } 
\eqfinp
\end{align*}

Third, by Proposition~\ref{pr:triplenorm_local_norms},
there exists a norm~\(\triplenorm{\cdot}\),
with unit ball 
\( \closedconvexhull\Bp{
\bigcup_{\scenario\in\SCENARIO} \BALL_{\scenario} } \).

This ends the proof.
\end{proof}

A possible choice for the
family \( \sequence{\theta_{\scenario}}{\scenario\in\SCENARIO} \)
of symmetric mappings
\( \theta_{\scenario} : \Uncertain_{\scenario} \to \PRIMAL \)
is given by the normalization mappings
\begin{equation*}
\forall \scenario\in\SCENARIO \eqsepv 
  \normalized_{\scenario}: 
\Uncertain_{\scenario}
  \to \SPHERE_{\scenario} \cup \{0\} 
  \eqsepv
  \normalized_{\scenario}\np{\primal}=
  \begin{cases}
    \frac{ \primal }{ \triplenorm{\primal}_{\scenario} }
    & \mtext{ if } \primal \in 
\Uncertain_{\scenario}\backslash\{0\} 
    \eqfinv 
    \\
    0 
    & \mtext{ if } \primal = 0 
\eqfinv
  \end{cases}
\end{equation*}
under the assumption (equivalent to~\eqref{eq:cor_triplenorm_SPHERE_BALL})
that 
\begin{equation}
  \overline{ \defset{ \frac{ \primal }{ \triplenorm{\primal}_{\scenario} } }%
{ \primal \in \Uncertain_{\scenario}\backslash\{0\} } } 
= \overline{ \normalized_{\scenario}\np{\Uncertain_{\scenario}} }
= \SPHERE_{\scenario} 
\eqsepv 
\forall \scenario\in\SCENARIO 
\eqfinp
\label{eq:cor_triplenorm_SPHERE}
\end{equation}

With these normalization mappings in Theorem~\ref{th:main},
we recover the case developed
in~\S\ref{Lower_bounds_for_exact_sparse_optimization_problems},
where sparsity is exactly measured by the \pseudonormlzero, with:
\begin{itemize}
\item 
finite set \( \SCENARIO=
\defset{\Scenario}{ \Scenario \subset \{1, \ldots, d \} 
\text{ and } \cardinal{\Scenario} \leq k } \),
\item 
subsets \( \Uncertain_{\Scenario}= 
\defset{ \primal \in \RR^d }{ \Support{\primal} = \Scenario } \)
of the Euclidian space \( \UNCERTAIN=\RR^d \) (sparsity),
for all \(  \scenario\in\SCENARIO \),
with the convention that \( \Uncertain_{\emptyset}=\{0\}\),
\item 
norms \( \triplenorm{\primal}_{\Scenario}=
\norm{\primal}_2 \) (amplitude).
\end{itemize}
Our framework encompasses the
(latent) group Lasso norms \cite{Obozinski-Jacob-Vert:inria-00628498,
Obozinski-Bach:hal-01412385}, with:
\begin{itemize}
\item 
finite set \( \SCENARIO=
2^{\{1, \ldots, d \}} = 
\defset{\Scenario}{ \Scenario \subset \{1, \ldots, d \} } \),
\item 
subsets \( \Uncertain_{\Scenario}= 
\defset{ \primal \in \RR^d }{ \Support{\primal} = \Scenario } \)
of the Euclidian space \( \UNCERTAIN=\RR^d \) (sparsity),
with the convention that \( \Uncertain_{\emptyset}=\{0\}\),
\item 
norms \( \triplenorm{\primal}_{\Scenario}=
\frac{ \norm{\primal}_q}{F\np{\Scenario}^{1/q}} \), 
where \( F : 2^{\{1, \ldots, d \}} \to ]0,+\infty] \)  (amplitude).
\end{itemize}
In both cases, a global norm is inferred, by convolution,
from local norms on~\( \closedspan\Uncertain_{\Scenario}= 
\RR^{\Scenario} \times \{0\}^{-\Scenario} \).
The condition~\eqref{eq:cor_triplenorm_SPHERE} holds indeed true
as \( \normalized_{\Scenario}\np{\Uncertain_{\Scenario}} \)
is \( \SPHERE_{\Scenario} \) minus a finite number of points
(those on the axis), hence 
\( \overline{ \normalized_{\Scenario}\np{\Uncertain_{\Scenario}} }
= \SPHERE_{\Scenario} \). 
Theorem~\ref{th:main} provides support for the use of this type of norms
in sparse optimization, and Proposition~\ref{pr:triplenorm_local_norms}
displays a general method to construct large classes of norms.

\renewcommand{\UNCERTAIN}{\mathbb{W}}
\renewcommand{\Uncertain}{W}
\renewcommand{\uncertain}{w}
\renewcommand{\PRIMAL}{{\mathbb X}}
\renewcommand{\Primal}{X}
\renewcommand{\primal}{x}
\renewcommand{\DUAL}{{\mathbb Y}}
\renewcommand{\Dual}{Y}
\renewcommand{\dual}{y}

\section{Conclusion}

In this paper, we have consider \emph{exact} sparse optimization problems,
that is, problems with combinatorial sparsity constraint.
More precisely, we have focused on
problems where sparsity is measured either by
the nonconvex \pseudonormlzero\ (and not by substitute penalty terms)
or by the belonging of the solution to a finite union of given subsets.

In exact sparse optimization problems, 
where sparsity is  measured by the \pseudonormlzero,
one looks for solution that have few nonzero components. 
It is well-known that the Fenchel biconjugate 
of the \pseudonormlzero\ is zero, making it hopeless to 
replace the \pseudonormlzero\ by its 
best lower convex lsc approximation.
In the same vein, the highly nonconvex constraint that the \pseudonormlzero\ 
be less than a given integer cannot be handled by the Fenchel conjugacy,
because the conjugate of its characteristic function is identically~$+\infty$.

In this paper, 
we have proposed to handle the \pseudonormlzero, not by the 
Fenchel conjugacy, but by a suitable so-called \Capra\ conjugacy,
as introduced in the companion paper~\cite{Chancelier-DeLara:2019b}.
By doing so, we have displayed a convex program that is a lower bound
of the original combinatorial optimization problem.
We insist that it is a lower bound, and not a substitute problem
with substitute penalty terms. Thus doing, we keep track of 
the original nonconvex problem.

Going on, we have studied
generalized sparse optimization, where the solution is looked for
in a finite union of given subsets.
We have identified suitable couplings, 
namely the one-sided linear couplings, and, 
thus equipped, we have been able to obtain more general results.
Our main contribution is to provide a systematic way to 
design norms, and associated convex programs 
that are lower bounds for the original 
exact sparse optimization problem.
\bigskip

\textbf{Acknowledgements.}
We want to thank 
Guillaume Obozinski, 
Yohann De Castro, 
Nathan De Lara, 
Antonin Chambolle,
Miguel Angel Goberna
and 
Marco Antonio Lopez Cerda
for discussions on first versions of this work.

\appendix
\section{Appendix}

\subsection{Background on J.~J. Moreau lower and upper additions}
\label{Moreau_lower_and_upper_additions}




When we manipulate functions with values 
in~$\bar\RR = [-\infty,+\infty] $,
we adopt the following Moreau \emph{lower addition} or
\emph{upper addition}, depending on whether we deal with $\sup$ or $\inf$
operations. 
We follow \cite{Moreau:1970}.
In the sequel, $u$, $v$ and $w$ are any elements of~$\bar\RR$.

\subsubsection*{Moreau lower addition}

\begin{subequations}
  The Moreau \emph{lower addition} extends the usual addition with 
  \begin{equation}
    \np{+\infty} \LowPlus \np{-\infty} = \np{-\infty} \LowPlus \np{+\infty} = -\infty \eqfinp
    \label{eq:lower_addition}
  \end{equation}
  With the \emph{lower addition}, \( \np{\bar\RR, \LowPlus } \) is a convex cone,
  with $ \LowPlus $ commutative and associative.
  The lower addition displays the following properties:
  \begin{align}
    u \leq u' \eqsepv v \leq v' 
    & 
      \Rightarrow u \LowPlus v \leq u'  \LowPlus v' 
\eqfinv 
\label{eq:lower_addition_leq}
\\
    (-u) \LowPlus (-v) 
    & 
      \leq -(u \LowPlus v) 
\eqfinv 
\label{eq:lower_addition_minus}
\\ 
    (-u) \LowPlus u 
    & 
      \leq 0 
\eqfinv 
\label{eq:lower_substraction_le_zero}
\\
    \sup_{a\in\AAA} f(a) \LowPlus \sup_{b\in\BB} g(b)
    &=
      \sup_{a\in\AAA,b\in\BB} \bp{f(a) \LowPlus g(b)} 
\eqfinv
\label{eq:lower_addition_sup}
\\ 
    \inf_{a\in\AAA} f(a) \LowPlus \inf_{b\in\BB} g(b) 
    & \leq 
      \inf_{a\in\AAA,b\in\BB} \bp{f(a) \LowPlus g(b)} 
\eqfinv
\label{eq:lower_addition_inf}
\\ 
    t < +\infty \Rightarrow    \inf_{a\in\AAA} f(a) \LowPlus t 
    & =
      \inf_{a\in\AAA} \bp{f(a) \LowPlus t} 
\eqfinp
\label{eq:lower_addition_inf_constant}
  \end{align}
\end{subequations}

\subsubsection*{Moreau upper addition}

\begin{subequations}
  The Moreau \emph{upper addition} extends the usual addition with 
  \begin{equation}
    \np{+\infty} \UppPlus \np{-\infty} = 
    \np{-\infty} \UppPlus \np{+\infty} = +\infty \eqfinp
    \label{eq:upper_addition}
  \end{equation}
  With the \emph{upper addition}, \( \np{\bar\RR, \UppPlus } \) is a convex cone,
  with $ \UppPlus $ commutative and associative.
  The upper addition displays the following properties:
  \begin{align}
    u \leq u' \eqsepv v \leq v' 
    & 
      \Rightarrow u \UppPlus v \leq u' \UppPlus v' \eqfinv
      \label{eq:upper_addition_leq} \\
    (-u) \UppPlus (-v) 
    & 
      \geq -(u \UppPlus v) \eqfinv
      \label{eq:upper_addition_minus} 
    \\ 
    (-u) \UppPlus u 
    & 
      \geq 0 \eqfinv
      \label{eq:upper_substraction_ge_zero} 
    \\ 
    \inf_{a\in\AAA} f(a) \UppPlus \inf_{b\in\BB} g(b) 
    & =
      \inf_{a\in\AAA,b\in\BB} \bp{f(a) \UppPlus g(b)} \eqfinv
      \label{eq:upper_addition_inf} 
    \\
    \sup_{a\in\AAA} f(a) \UppPlus \sup_{b\in\BB} g(b)
    & \geq 
      \sup_{a\in\AAA,b\in\BB} \bp{f(a) \UppPlus g(b)} \eqfinv
      \label{eq:upper_addition_sup}
    \\ 
    -\infty < t \Rightarrow    \sup_{a\in\AAA} f(a) \UppPlus t 
    & =
      \sup_{a\in\AAA} \bp{f(a) \UppPlus t} \eqfinp
      \label{eq:upper_addition_sup_constant}
  \end{align}
\end{subequations}

\subsubsection*{Joint properties of the Moreau lower and upper addition}

\begin{subequations}
We obviously have that
\begin{equation}
  u \LowPlus v \leq u \UppPlus v \eqfinp 
\label{eq:lower_leq_upper_addition}
\end{equation}
The Moreau lower and upper additions are related by
\begin{equation}
-(u \UppPlus v) = (-u) \LowPlus (-v) \eqsepv 
-(u \LowPlus v) = (-u) \UppPlus (-v) \eqfinp 
\label{eq:lower_upper_addition_minus}  
\end{equation}
They satisfy the inequality
\begin{equation}
(u \UppPlus v) \LowPlus w \leq u \UppPlus (v \LowPlus w) \eqfinp
                                 \label{eq:lower_upper_addition_inequality} 
\end{equation}
with 
  \begin{equation}
(u \UppPlus v) \LowPlus w < u \UppPlus (v \LowPlus w) 
\iff
\begin{cases}
 u=+\infty \mtext{ and } w=-\infty \eqsepv \\
\mtext{ or } \\
u=-\infty \mtext{ and } w=+\infty 
\mtext{ and } -\infty < v < +\infty \eqfinp
\end{cases}
                                 \label{eq:lower_upper_addition_equality} 
\end{equation}
Finally, we have that 
\begin{align}
& u \LowPlus (-v) \leq 0 \iff u \leq v  \iff  0 \leq v \UppPlus (-u) 
\eqfinv
\label{eq:lower_upper_addition_comparisons}
\\
& u \LowPlus (-v) \leq w \iff u \leq v \UppPlus w \iff u \LowPlus (-w) \leq v
\eqfinv 
\\  
& w \leq v \UppPlus (-u) \iff u \LowPlus w \leq v \iff u \leq v \UppPlus (-w) 
\eqfinp
\end{align}
 \end{subequations}



\subsection{Background on sets and functions}
\label{Background_on_sets_and_functions}

\renewcommand{\UNCERTAIN}{\mathbb{W}}
\renewcommand{\Uncertain}{W}
\renewcommand{\uncertain}{w}
\renewcommand{\PRIMAL}{{\mathbb X}}
\renewcommand{\Primal}{X}
\renewcommand{\primal}{x}
\renewcommand{\DUAL}{{\mathbb Y}}
\renewcommand{\Dual}{Y}
\renewcommand{\dual}{y}

Let $\UNCERTAIN$ be a set.
\begin{itemize}
\item 
The \emph{effective domain} of a function 
\( \fonctionuncertain : \UNCERTAIN \to \barRR \)
is \( \dom\fonctionuncertain = 
\defset{\uncertain\in\UNCERTAIN}{\fonctionuncertain\np{\uncertain}<+\infty} \).
\item 
The function \( \fonctionuncertain : \UNCERTAIN \to \barRR \)
is said to be \emph{proper} when 
\( \dom\fonctionuncertain \neq \emptyset \) and 
\( \defset{\uncertain\in\UNCERTAIN}{\fonctionuncertain\np{\uncertain}=-\infty} 
=~\emptyset \).
\item 
When \( \UNCERTAIN \) is a topological space, 
\( \continuitypoints\fonctionuncertain \) denotes the \emph{continuity points}
of the function \( \fonctionuncertain : \UNCERTAIN \to \barRR \).
\item 
The \emph{characteristic function} 
of a set~$\Uncertain \subset \UNCERTAIN$ is the function~$\delta_{\Uncertain}$ 
defined by 
\begin{equation}
  \delta_{\Uncertain}\np{\uncertain} =  
   \begin{cases}
      0 &\text{ if } \uncertain \in \Uncertain \eqfinv \\ 
      +\infty &\text{ if } \uncertain \not\in \Uncertain \eqfinp
    \end{cases} 
\label{eq:characteristic_function} 
\end{equation}
\end{itemize}

\subsubsection{Topological vector space}

Let $\UNCERTAIN$ be a topological vector space
and let $\Uncertain \subset \UNCERTAIN$.
\begin{itemize}
\item 
The set~\( \Uncertain \) is \emph{symmetric} if
\( -\Uncertain=\Uncertain \).
\item 
The \emph{conical hull} of~$\Uncertain$ is the smallest cone in~$\UNCERTAIN$
that contains~$\Uncertain$, denoted by \( \conicalhull\Uncertain \).
\item 
The \emph{convex hull} of~$\Uncertain$ is the smallest convex set in~$\UNCERTAIN$
that contains~$\Uncertain$, denoted by \( \convexhull\Uncertain \).
\item 
The \emph{closed convex hull} of~$\Uncertain$ is the smallest closed convex set in~$\UNCERTAIN$
that contains~$\Uncertain$, denoted by \( \closedconvexhull\Uncertain \).
\item 
The \emph{span} of~$\Uncertain$ is the smallest
subspace of~$\UNCERTAIN$
that contains~$\Uncertain$, denoted by \( \linearspan\Uncertain \).
\item 
The \emph{closed span} of~$\Uncertain$ is the smallest closed 
subspace of~$\UNCERTAIN$
that contains~$\Uncertain$, denoted by \( \closedspan\Uncertain \).
\end{itemize}

\subsubsection{Dual system}
\label{Dual_system}

We say that two vector spaces $\PRIMAL$ and $\DUAL$ 
form a \emph{dual system} \cite[p. 211]{Aliprantis-Border:1999}
if $\PRIMAL$ and $\DUAL$ are equipped with 
a bilinear form \( \proscal{}{} \), 
such that 
\( \Bp{ \forall \primal \in \PRIMAL \eqsepv
\proscal{\primal}{\dual} =0 } \Rightarrow \dual=0 \)
and
\( \Bp{ \forall \dual \in \DUAL \eqsepv
\proscal{\primal}{\dual} =0 } \Rightarrow \primal=0 \).
By default, the primal space~$\PRIMAL$ is equipped with the 
weak topology (of pointwise convergence), and the dual space~ $\DUAL$ 
with the weak* topology.

\begin{definition}
Let $\PRIMAL$ and $\DUAL$ be a dual system
and let \( \Primal \subset \PRIMAL \).
\begin{itemize}
\item 
The \emph{support function} 
of~$\Primal$ is defined by 
\begin{equation}
  \sigma_{\Primal}\np{\dual} = 
\sup_{\primal\in\Primal} \proscal{\primal}{\dual}
\eqsepv \forall \dual \in \DUAL
\eqfinp
\label{eq:support_function}
\end{equation}
\item 
The \emph{barrier cone} of~$\Primal$
is the effective domain of the support function~$\sigma_{\Primal}$:
\begin{equation}
  \barriercone\Primal
= \defset{\dual \in \DUAL}%
{ \sup_{\primal\in\Primal} \proscal{\primal}{\dual} < +\infty }
=\dom\sigma_{\Primal} 
    \eqfinp
\label{eq:barrier_cone}
\end{equation}
\item 
The set \( \Primal \) is said to be \emph{weakly bounded}
if \( \sup_{\primal\in\Primal} \proscal{\primal}{\dual} < +\infty \)
for all \( \dual \in \DUAL \):
\begin{equation}
\Primal \text{ is weakly bounded } \iff 
\barriercone\Primal=\DUAL \iff 
\dom\sigma_{\Primal} =\DUAL
\eqfinp
\label{eq:weakly_bounded}
\end{equation}
\item 
The \emph{orthogonal set} of \( \Primal \) is defined by 
\begin{equation}
  \Primal^{\perp}
= \defset{\dual \in \DUAL}{\proscal{\primal}{\dual} =0
      \eqsepv \forall \primal \in \Primal } 
\eqfinp
\label{eq:orthogonal_set}
\end{equation}
\item 
The \emph{polar set} of \( \Primal \) is defined by 
\begin{equation}
  \Primal^{\odot} 
= \defset{\dual \in \DUAL}{\proscal{\primal}{\dual} \leq 1 
      \eqsepv \forall \primal \in \Primal } 
\eqfinp
\label{eq:polar_set}
\end{equation}
\end{itemize}
We obtain symmetric definitions for~\( \Dual \subset \DUAL \).
\label{de:dual_system}
\end{definition}


We provide different properties of barrier cones
and of weakly bounded sets.

\begin{proposition}
\quad
  \begin{enumerate}
  \item 
  The barrier cone~\eqref{eq:barrier_cone} of~$\Primal\subset\PRIMAL$
satisfies the following properties
\begin{subequations}
\begin{align}
  \barriercone\Primal 
&= 
\barriercone\np{\convexhull\Primal}
=\barriercone\np{\closedconvexhull\Primal}
    \eqfinv
\label{eq:barrier_cone_and_convexhull}
\\
  \barriercone\Primal 
&=
\conicalhull\bp{\Primal^{\odot}}
\label{eq:barrier_cone=cone_generated_by_polar}
\eqfinp 
\end{align}
\end{subequations}
\item 
Let 
\( \sequence{\Primal_{\scenario}}{\scenario\in\SCENARIO} \)
be a family of subsets of~$\PRIMAL$. Then 
\begin{equation}
\SCENARIO \textrm{ is a finite set } \Rightarrow
 \barriercone\bp{\bigcup_{\scenario\in\SCENARIO} \Primal_{\scenario}}
=\bigcap_{\scenario\in\SCENARIO} \barriercone\Primal_{\scenario}
\eqfinp
\label{eq:barrier_cone_of_a_finite_union}
\end{equation}
As an application, if \( \sequence{\Primal_{\scenario}}{\scenario\in\SCENARIO} \)
is a finite family of weakly bounded subsets of~$\PRIMAL$, then 
 the finite union \( \bigcup_{\scenario\in\SCENARIO} \Primal_{\scenario} \) 
is weakly bounded.
\label{it:barrier_cone_of_a_finite_union}
\item 
If $\HILBERT$ is a Hilbert space, 
then bounded subsets of~$\HILBERT$ are weakly bounded:
\begin{equation}
\Hilbert \subset \HILBERT \textrm{ is bounded } \Rightarrow
 \barriercone\Hilbert=\HILBERT
\eqfinp
\label{eq:barrier_cone_of_a_bounded_set}
\end{equation}
  \end{enumerate}
\label{pr:barrier_cone}
\end{proposition}

\begin{proof}
  \begin{enumerate}
  \item 
Equation~\eqref{eq:barrier_cone_and_convexhull}
is a consequence of the definition~\eqref{eq:barrier_cone}
and of the property of support functions that 
\( \sigma_{ \Primal } = \sigma_{ \convexhull\Primal }
= \sigma_{ \closedconvexhull\Primal } \). 

The proof of~\eqref{eq:barrier_cone=cone_generated_by_polar}
follows easily from
\[
\barriercone\Primal = 
\bigcup_{\lambda >0} \defset{\dual \in \DUAL}%
{ \sup_{\primal\in\Primal} \proscal{\primal}{\dual} \leq \lambda }
=\bigcup_{\lambda >0} \defset{\dual \in \DUAL}%
{\frac{\dual}{\lambda} \in \Primal^{\odot}} 
=\conicalhull\bp{\Primal^{\odot}}
    \eqfinv
\]
by the definition~\eqref{eq:polar_set} of the 
polar set~\( \Primal^{\odot} \).

\item 
The proof of~\eqref{eq:barrier_cone_of_a_finite_union} follows
from the observation that 
\( \sigma_{ \np{\bigcup_{\scenario\in\SCENARIO} \Primal_{\scenario}} }
= \max_{\scenario\in\SCENARIO}\sigma_{ \Primal_{\scenario} } \)
with a maximum since the set $\SCENARIO$ is finite,
and from the definition~\eqref{eq:barrier_cone}
that \( \barriercone\Primal =\dom\sigma_{\Primal} \).

As an application, if every \( \Primal_{\scenario} \) is weakly bounded,
for every \( \scenario\in\SCENARIO \),
and $\SCENARIO$ is finite, we get, by~\eqref{eq:barrier_cone_of_a_finite_union}, that 
\[ 
\barriercone\Bp{ \bigcup_{\scenario\in\SCENARIO} \Primal_{\scenario} }
= \bigcap_{\scenario\in\SCENARIO} \barriercone\Primal_{\scenario}
= \bigcap_{\scenario\in\SCENARIO} \DUAL = \DUAL 
\eqfinv
\]
and we conclude that the finite union 
\( \bigcup_{\scenario\in\SCENARIO} \Primal_{\scenario} \) 
is weakly bounded by definition~\eqref{eq:barrier_cone}.
\item 

\renewcommand{\PRIMAL}{\HILBERT}
\renewcommand{\DUAL}{\HILBERT}
\renewcommand{\Primal}{\Hilbert}
\renewcommand{\Dual}{\Hilbert'}
\renewcommand{\primal}{\hilbert}
\renewcommand{\dual}{\hilbert'}
The proof of~\eqref{eq:barrier_cone_of_a_bounded_set} follows
from the observation that, in a Hilbert space, 
\( \proscal{\primal}{\dual} \leq \norm{\primal} \norm{\dual} \), so that,
for any \( \dual \in \DUAL \), we have that 
\(
\sup_{\primal\in\Primal} \proscal{\primal}{\dual} \leq 
\bp{ \sup_{\primal\in\Primal} \norm{\primal} } \norm{\dual} <+\infty
\), as \( \sup_{\primal\in\Primal} \norm{\primal} <+\infty \)
since $\Primal$ is bounded.

  \end{enumerate}

This ends the proof.  
\end{proof}

\subsection{Background on norms and dual norms}
\label{Dual_norm}


Here, we collect different results on norms, equivalent norms,  
and norms induced by support functions (in a dual system and in 
a Hilbert space).
\medskip

For a norm~\( \norm{\cdot} \) on a vector space~$\UNCERTAIN$,
we denote the unit ball by
\begin{equation}
  \BALL_{\norm{\cdot}}
= \defset{\uncertain \in \UNCERTAIN}{\norm{\uncertain} \leq 1} 
 \eqfinp 
    \label{eq:NORM_UNIT_BALL}
\end{equation}
A unit ball is always convex, symmetric and 
with full conical hull, that is, 
\( \conicalhull\BALL_{\triplenorm{\cdot}}=\UNCERTAIN\)
(indeed any \( \uncertain\in\UNCERTAIN\backslash\{0\} \) can be written
\( \uncertain=\norm{\uncertain} \frac{\uncertain}{\norm{\uncertain}}
\in \conicalhull\BALL_{\norm{\cdot}}\), 
and \( 0\in \BALL_{\norm{\cdot}} \subset
\conicalhull\BALL_{\norm{\cdot}}\)).

\subsubsection{Equivalent norms}

We recall definition and characterizations of 
equivalent norms.

\begin{proposition}
  Let \( \norm{\cdot}^\sharp \) and 
  \( \norm{\cdot}^\flat \) be two norms on a vector space~$\UNCERTAIN$.
  The following statements are equivalent.
  \begin{enumerate}
  \item \label{p:un}
    There exists $M>0$ such that 
    \( \norm{\cdot}^\flat \leq M \norm{\cdot}^\sharp \).
  \item \label{p:deux}
    The topology of~\( \norm{\cdot}^\sharp \) is richer 
    (contains more open sets) than 
    the topology of~\( \norm{\cdot}^\flat \).
  \item \label{p:trois}
    The function 
    \( \norm{\cdot}^\flat : \bp{\UNCERTAIN,\norm{\cdot}^\sharp} \to \barRR \)
    is continuous.
  \item \label{p:quatre}
    The unit ball \( \BALL_{\norm{\cdot}^\flat} \) is closed 
    for the topology of \( \norm{\cdot}^\sharp \).
  \item \label{p:cinq}
    The unit ball \( \BALL_{\norm{\cdot}^\flat} \) has nonempty interior
    for the topology of \( \norm{\cdot}^\sharp \). 
  \item \label{p:six}
    \( 0 \in \mathrm{int}_{\norm{\cdot}^\sharp}\BALL_{\norm{\cdot}^\flat} \). 
  \item \label{p:sept}
    The unit ball \( \BALL_{\norm{\cdot}^\sharp} \) 
    is bounded for the norm~\( \norm{\cdot}^\flat \).
  \end{enumerate}
  \label{pr:comparison_of_norms}
\end{proposition}

\begin{proof} 
The chain of implications (in both directions) from statements~\ref{p:un} to
\ref{p:quatre}
is easy to prove. So is statement~\ref{p:trois} $\Rightarrow$ 
 statement~\ref{p:cinq}.

Statement~\ref{p:sept} is equivalent to the property that 
    there exists $M>0$ such that 
\( \BALL_{\norm{\cdot}^\sharp} \subset M\BALL_{\norm{\cdot}^\flat} \),
hence to statement~\ref{p:un}.

Statement~\ref{p:six} is equivalent to the property that
    there exists $M>0$ such that $\frac{1}{M}\BALL_{\norm{\cdot}^\sharp} 
\subset \BALL_{\norm{\cdot}^\flat}$, hence is equivalent to statement~\ref{p:sept}.
Indeed, using~\cite[(6.6) p.~114]{Bauschke-Combettes:2017}, we have that
the interior of a set~$D$ is 
\( 
    \mathrm{int}_{\norm{\cdot}^\sharp}D = \defset{ \uncertain\in D}
    { \np{ \exists \rho > 0}\quad \rho \BALL_{\norm{\cdot}^\sharp} 
\subset D - \uncertain} 
\). 
With this, we prove that statement~\ref{p:cinq} implies
statement~\ref{p:six} (the reverse is obvious).
Let \( \uncertain\in 
\mathrm{int}_{\norm{\cdot}^\sharp}\BALL_{\norm{\cdot}^\flat}
=
\defset{ \uncertain\in \BALL_{\norm{\cdot}^\flat}}%
    { \np{ \exists \rho > 0}\quad \rho\BALL_{\norm{\cdot}^\sharp} 
\subset \BALL_{\norm{\cdot}^\flat} - \uncertain} \), 
there exists $\rho >0$ such that \(\rho
\BALL_{\norm{\cdot}^\sharp} \subset \BALL_{\norm{\cdot}^\flat}  - \uncertain\). 
Now, choosing $\mu = 1/(1+ \norm{\uncertain}^\flat)$, we get that $\mu \rho
\BALL_{\norm{\cdot}^\sharp} \subset \mu\np{\BALL_{\norm{\cdot}^\flat}  -
  \uncertain}\subset \BALL_{\norm{\cdot}^\flat}$,
and thus \( 0 \in \mathrm{int}_{\norm{\cdot}^\sharp}\BALL_{\norm{\cdot}^\flat}
\). 

This ends the proof.
\end{proof}

We easily deduce the following Proposition
(and the definition of equivalent norms).

\begin{proposition}
  Let \( \norm{\cdot}^\sharp \) and 
  \( \norm{\cdot}^\flat \) be two norms on a vector space~$\UNCERTAIN$.
  The following statements are equivalent.
  \begin{enumerate}
  \item 
There exist two positive numbers $m$ and $M$,
  such that 
  \begin{equation}
    0 < m \leq M < +\infty 
    \mtext{ and }
    m \norm{\cdot}^\sharp \leq \norm{\cdot}^\flat \leq M \norm{\cdot}^\sharp 
    \eqfinp
    \label{eq:norm-equiv}
  \end{equation}
\item 
The topologies of \( \norm{\cdot}^\sharp \) and
\( \norm{\cdot}^\flat \) are the same.
\item 
   The unit ball \( \BALL_{\norm{\cdot}^\flat} \) is closed 
    for the topology of \( \norm{\cdot}^\sharp \),
and bounded for the norm~\( \norm{\cdot}^\sharp \).
\item 
   The unit ball \( \BALL_{\norm{\cdot}^\flat} \) is closed 
    for the topology of \( \norm{\cdot}^\sharp \),
and \( 0 \in \mathrm{int}_{\norm{\cdot}^\sharp}\BALL_{\norm{\cdot}^\flat} \). 
\item 
   The unit ball \( \BALL_{\norm{\cdot}^\sharp} \) is closed 
    for the topology of \( \norm{\cdot}^\flat \),
and bounded for the norm~\( \norm{\cdot}^\flat \).
\item 
   The unit ball \( \BALL_{\norm{\cdot}^\sharp} \) is closed 
    for the topology of \( \norm{\cdot}^\flat \),
and \( 0 \in \mathrm{int}_{\norm{\cdot}^\flat}\BALL_{\norm{\cdot}^\sharp} \). 
  \end{enumerate}
In any of these equivalent cases, we say that the norms
\( \norm{\cdot}^\sharp \) and \( \norm{\cdot}^\flat \) 
are \emph{equivalent}.
  \label{pr:equivalent_norms}
\end{proposition}

\subsubsection{Dual norm in the dual system case}

Let $\PRIMAL$ and $\DUAL$ be a dual system, as recalled
in~\S\ref{Background_on_sets_and_functions}.
By default, the primal space~$\PRIMAL$ is equipped with the 
weak topology (of pointwise convergence), and the dual space~ $\DUAL$ 
with the weak* topology.
In the paper, we will mostly consider the 
case where $\PRIMAL=\DUAL$ is a Hilbert space, 
and the natural dual system it induces.

We study under which stronger and stronger assumptions
the support function of a set is a norm. 

\begin{proposition}
Let $\PRIMAL$ and $\DUAL$ be a dual system.
\begin{enumerate}
\item 
Let  \( \Primal \subset \PRIMAL \) be symmetric, weakly bounded and with full
conical hull, that is, such that
\begin{equation}
  -\Primal=\Primal 
\eqsepv
\barriercone\Primal = \DUAL 
\eqsepv 
\conicalhull\Primal= \PRIMAL  
\eqfinp
\label{eq:conditions_on_Primal_support_function=norm}
\end{equation}
Then the support function~\( \sigma_{\Primal} \) is a norm on~$\DUAL$,
whose unit ball is \( \Primal^{\odot} \).
\label{it:conditions_on_Primal_support_function=norm}
\item 
Let \( \Convex \subset \PRIMAL \) be closed, convex and containing~$0$.
  The following statements are equivalent.
\begin{enumerate}
\item 
The support function~\( \sigma_{\Convex} \) is a norm on~$\DUAL$,
whose unit ball is the polar set~\( \Convex^{\odot} \).
\label{it:equivalences_support_function=norm_support_function_set}
\item 
The set \( \Convex \) is 
symmetric, weakly bounded and with full conical hull.
\label{it:equivalences_support_function=norm_set}
\item 
The polar set \( \Convex^{\odot} \) is 
symmetric, weakly bounded and with full conical hull.
\label{it:equivalences_support_function=norm_polar_set}
\item 
The support function~\( \sigma_{\Convex^{\odot}} \) is a norm on~$\PRIMAL$,
whose unit ball is \( \Convex \).
\label{it:equivalences_support_function=norm_support_function_polar_set}
\end{enumerate}
\label{it:equivalences_support_function=norm_support_function_closed}
\end{enumerate}
\label{pr:conditions_on_Primal_support_function=norm}
\end{proposition}

\begin{proof}
\begin{enumerate}
\item 
We prove item~\ref{it:conditions_on_Primal_support_function=norm}.

First, as \( \Primal \) is weakly bounded, that is, 
\( \barriercone\Primal = \DUAL \), we have that 
\( \dom\sigma_{\Primal} =\DUAL\) by~\eqref{eq:weakly_bounded},
hence that \( \sigma_{\Primal} < +\infty \).

Second, as \( \Primal \) is symmetric, that is, 
\( -\Primal=\Primal \), we have that 
\( \sigma_{\Primal}\np{\dual}=\sigma_{\Primal}\np{-\dual} \),
for all \( \dual \in \DUAL \).

Third, as \( \Primal \) is symmetric (and nonempty since 
 \( \conicalhull\Primal= \PRIMAL \)),
we deduce that \( 0 \in \convexhull\Primal \), hence that 
\( \sigma_{\Primal}\np{\dual}=
\sigma_{ \convexhull\Primal }\np{\dual}\geq 0\),
for all \( \dual \in \DUAL \).

Fourth, we show that \( \sigma_{\Primal}\np{\dual}=0 
\Rightarrow \dual=0 \).
Indeed, from \( \sigma_{\Primal} \geq 0\), we deduce that 
\( \sigma_{\Primal}\np{\dual}=0 \iff \dual \in \Primal^\perp \).
Now, as \( \Primal \) has full conical hull, that is, 
\( \conicalhull\Primal= \PRIMAL \), we deduce that 
\( \Primal^\perp = \bp{\conicalhull\Primal}^\perp 
= \PRIMAL^\perp =\{0\} \), hence \( \dual=0 \).

Finally, we conclude that \( \sigma_{\Primal} \) is a norm since it is 
subadditive and 1-homogeneous, as it is a support function.

The unit ball of the norm~\( \sigma_{\Primal} \) is 
\( \BALL_{\sigma_{\Primal}}=
\defset{\dual \in \DUAL}{ \sigma_{\Primal}\np{\dual} \leq 1} 
= \Primal^{\odot} \) by definition~\eqref{eq:polar_set} 
of the polar set of~\( \Primal \).

\item 
We prove item~\ref{it:equivalences_support_function=norm_support_function_closed}.

Since the set \( \Convex \) is closed, convex and contains~$0$,
we have \( \Convex^{\odot\odot} = \Convex \) 
\cite[Th.~5.103]{Aliprantis-Border:1999}.

  \begin{itemize}
  \item 

  We prove that 
statement~\ref{it:equivalences_support_function=norm_support_function_set} 
implies
statement~\ref{it:equivalences_support_function=norm_set}.

The set \( \Convex \) is symmetric because 
\( \sigma_{\Convex}\np{\dual}=\sigma_{\Convex}\np{-\dual} \),
for all \( \dual \in \DUAL \), implies that 
\( \sigma_{\Convex}=\sigma_{-\Convex} \), hence that 
\( -\Convex=\Convex \) since \( \Convex \) is closed and convex.
The set \( \Convex \) is weakly bounded because 
\( \sigma_{\Convex} < +\infty \iff \dom\sigma_{\Convex} =\DUAL
\iff \barriercone\Convex = \DUAL \) by~\eqref{eq:weakly_bounded}.
The set \( \Convex \) has full conical hull because 
\( \dual \in \bp{\linearspan\Convex}^\perp=\Convex^\perp
\Rightarrow \sigma_{\Convex}\np{\dual}=0 
\Rightarrow \dual=0 \), hence 
\( \linearspan\Convex=\PRIMAL \); now, as the set \( \Convex \) 
is convex and symmetric, we have that \( \linearspan\Convex=
\conicalhull\Convex \).

\item 
By item~\ref{it:conditions_on_Primal_support_function=norm}, 
statement~\ref{it:equivalences_support_function=norm_set}
implies
statement~\ref{it:equivalences_support_function=norm_support_function_set}.

\item 
  We prove that 
statement~\ref{it:equivalences_support_function=norm_set}
implies
statement~\ref{it:equivalences_support_function=norm_polar_set}.

The conditions~\eqref{eq:conditions_on_Primal_support_function=norm}
give
\begin{subequations}
\begin{align}
    -\np{\Convex^{\odot}}
&=
\np{-\Convex}^{\odot}=\Convex^{\odot}
\eqfinv
\intertext{ by \( -\Convex=\Convex \) and by definition~\eqref{eq:polar_set} 
of the polar set,}
\conicalhull\np{\Convex^{\odot}}
&= 
\barriercone\Convex = \DUAL 
\eqfinv
\intertext{ by~\eqref{eq:barrier_cone=cone_generated_by_polar}
and the assumption that \( \Convex \) weakly bounded,}
\barriercone\np{\Convex^{\odot}}
&=
\conicalhull\np{\Convex^{\odot\odot}}
=\conicalhull\Convex=\PRIMAL
\eqfinv
\end{align}
by \( \Convex^{\odot\odot} = \Convex \) and since
 \( \Convex \) has full conical hull.
\label{eq:conditions_on_Primal_support_function=norm_polar}
\end{subequations}

\item 
Statement~\ref{it:equivalences_support_function=norm_polar_set}
implies
statement~\ref{it:equivalences_support_function=norm_set}.
Indeed, we use the shown property that 
statement~\ref{it:equivalences_support_function=norm_set}
implies
statement~\ref{it:equivalences_support_function=norm_polar_set},
but with \( \Convex^{\odot} \) instead of~\( \Convex \),
where the polar set \( \Convex^{\odot} \) is closed convex
and contains~$0$. Thus, we obtain 
statement~\ref{it:equivalences_support_function=norm_set}
for \( \Convex^{\odot\odot} \), but we have seen that
\( \Convex^{\odot\odot} = \Convex \).

\item 
Because the polar set \( \Convex^{\odot} \) is closed convex and contains~$0$ ,
we deduce that 
statement~\ref{it:equivalences_support_function=norm_polar_set}
is equivalent to
statement~\ref{it:equivalences_support_function=norm_support_function_polar_set}
from the shown property that 
statement~\ref{it:equivalences_support_function=norm_support_function_set} 
is equivalent to
statement~\ref{it:equivalences_support_function=norm_set}.
  \end{itemize}
\end{enumerate}

This ends the proof.
\end{proof}

Now, we define the dual norm.
\begin{definition}
Let $\PRIMAL$ and $\DUAL$ be a dual system.
Let \( \triplenorm{\cdot} \) be a norm on~$\PRIMAL$.
If the support function~$\sigma_{\BALL_{\triplenorm{\cdot}}} $
is a norm (on~$\DUAL$), it is called the \emph{dual norm} 
of~\( \triplenorm{\cdot} \) and it is denoted by~\( \triplenorm{\cdot}_\star \).
  \label{de:DUAL_NORM}
\end{definition}
When a dual norm exists~\( \triplenorm{\cdot}_\star \), then,
by item~\ref{it:conditions_on_Primal_support_function=norm}
in Proposition~\ref{pr:conditions_on_Primal_support_function=norm},
its unit ball is the polar set of the original unit ball:
  \begin{equation}
\BALL_{\triplenorm{\cdot}_\star}
= \BALL_{\triplenorm{\cdot}}^{\odot} 
     \eqfinp
\label{eq:norm_and dual_norm_unit_ball}
  \end{equation}

When both the norm~\( \triplenorm{\cdot} \) 
and the dual norm~\( \triplenorm{\cdot}_\star \) admit a dual norm,
the norm~\( \triplenorm{\cdot}_{\star\star} = 
\bp{\triplenorm{\cdot}_\star}_\star \) (on~$\PRIMAL$)
is called the \emph{bidual norm}.
We provide a characterization of when a dual norm exists,
and of when a bidual norm exists and coincides with the original norm.

\begin{proposition}
Let $\PRIMAL$ and $\DUAL$ be a dual system.
Let \( \triplenorm{\cdot} \) be a norm on~$\PRIMAL$.
\begin{enumerate}
\item 
  The following statements are equivalent.
  \begin{enumerate}
  \item 
The norm~\( \triplenorm{\cdot} \) admits a dual norm.
  \label{it:norm_and_dual_norm_dual_system_dual_norm_un}
\item 
The unit ball~$\BALL_{\triplenorm{\cdot}}$ is weakly bounded.
  \label{it:norm_and_dual_norm_dual_system_dual_norm_deux}
  \end{enumerate}
  \label{it:norm_and_dual_norm_dual_system_dual_norm}
\item 
  The following statements are equivalent.
  \begin{enumerate}
  \item 
The norm \( \triplenorm{\cdot} \) admits 
a dual norm~\( \triplenorm{\cdot}_\star \),
and the dual norm~\( \triplenorm{\cdot}_\star \)
has \( \triplenorm{\cdot} \) for  dual norm
(\( \triplenorm{\cdot}_{\star\star} =\triplenorm{\cdot} \)). 
  \label{it:norm_and_dual_norm_dual_system_bidual_norm_un}
\item 
The unit ball~$\BALL_{\triplenorm{\cdot}}$ is weakly bounded and closed.
  \label{it:norm_and_dual_norm_dual_system_bidual_norm_deux}
  \end{enumerate}
  \begin{subequations}
    In that case, each norm is the dual norm of the other norm,
  the unit balls
  \( \BALL_{\triplenorm{\cdot}} \) 
  and
  \( \BALL_{\triplenorm{\cdot}_\star} \) 
  are polar to each other, that is, 
  \begin{equation}
    \BALL_{\triplenorm{\cdot}} = \BALL_{\triplenorm{\cdot}_\star}^{\odot} 
    \mtext{ and }
    \BALL_{\triplenorm{\cdot}_\star}= \BALL_{\triplenorm{\cdot}}^{\odot} 
    \eqfinv
    \label{eq:norm_dual_norm_polar_balls_Hilbert}
  \end{equation}
  and their support functions satisfy 
  \begin{equation}
    \triplenorm{\cdot} = \sigma_{ \BALL_{\triplenorm{\cdot}_{\star}}}
    \mtext{ and }
    \triplenorm{\cdot}_\star = \sigma_{\BALL_{\triplenorm{\cdot}}} 
    \eqfinp 
    \label{eq:norm_dual_norm=support_of_polar_balls_Hilbert}
  \end{equation}
  \end{subequations}
  \label{it:norm_and_dual_norm_dual_system_bidual_norm}
\end{enumerate}
  \label{pr:norm_and_dual_norm_dual_system}
\end{proposition}

\begin{proof}

  \begin{enumerate}
\item 
We prove item~\ref{it:norm_and_dual_norm_dual_system_dual_norm}.

\begin{itemize}
\item 
We prove that 
statement~\ref{it:norm_and_dual_norm_dual_system_dual_norm_un}
implies
statement~\ref{it:norm_and_dual_norm_dual_system_dual_norm_deux}.

Indeed, if the norm~\( \triplenorm{\cdot} \) admits a dual norm, 
the support function~$\sigma_{\BALL_{\triplenorm{\cdot}}} $ satisfies
\( \sigma_{\BALL_{\triplenorm{\cdot}}} < +\infty \). Therefore 
\( \dom\sigma_{\BALL_{\triplenorm{\cdot}}} =\DUAL \), 
meaning that the unit ball~$\BALL_{\triplenorm{\cdot}}$ is weakly bounded 
by~\eqref{eq:weakly_bounded}.

\item 
We prove that 
statement~\ref{it:norm_and_dual_norm_dual_system_dual_norm_deux}
implies
statement~\ref{it:norm_and_dual_norm_dual_system_dual_norm_un}

Indeed,  being a unit ball, \( \BALL_{\triplenorm{\cdot}} \) is 
convex, symmetric and with full conical hull.
Moreover, it is also weakly bounded by assumption.
We deduce from item~\ref{it:conditions_on_Primal_support_function=norm}
in Proposition~\ref{pr:conditions_on_Primal_support_function=norm}
that the support function~\( \sigma_{\BALL_{\triplenorm{\cdot}}} \) is a norm on~$\DUAL$,
whose unit ball is the polar set~\( \BALL_{\triplenorm{\cdot}}^{\odot} \).
\end{itemize}

\item 
Item~\ref{it:norm_and_dual_norm_dual_system_bidual_norm}
is a straightforward consequence of
item~\ref{it:equivalences_support_function=norm_support_function_closed}
in Proposition~\ref{pr:conditions_on_Primal_support_function=norm}
with \( \Convex= \BALL_{\triplenorm{\cdot}} \).
Indeed, being a unit ball, \( \BALL_{\triplenorm{\cdot}} \) is 
convex, containing~$0$, symmetric and with full conical hull.
Moreover, it is also closed and weakly bounded by assumption.

The equations 
\eqref{eq:norm_dual_norm_polar_balls_Hilbert}--\eqref{eq:norm_dual_norm=support_of_polar_balls_Hilbert}
are also a straightforward consequence of
item~\ref{it:equivalences_support_function=norm_support_function_closed}
in Proposition~\ref{pr:conditions_on_Primal_support_function=norm}.
  \end{enumerate}

This ends the proof. 
\end{proof}

\subsubsection{Dual norm in the Hilbert space case}

Let $\PRIMAL=\DUAL$ be a Hilbert space 
with scalar product \( \proscal{}{} \), 
and induced \emph{Hilbertian norm}
\( \norm{\cdot} = \sqrt{ \proscal{\cdot}{\cdot} } \)
and Hilbertian topology.
It is easy to see that the dual norm~$\norm{\cdot}_\star$ of the Hilbertian norm
is the Hilbertian norm, that is, $\norm{\cdot}_\star=\norm{\cdot}$.

When we refer to notions attached to a dual system
(support function, weakly bounded set), 
by default it is the natural dual system induced by the 
Hilbertian structure.

We study under which assumptions
the support function of a set is a norm,
and the topology that it induces. 

\begin{proposition}
Let  \( \Convex \subset \PRIMAL \) be closed, convex, 
symmetric, weakly bounded and with full conical hull
($\conicalhull\Convex= \PRIMAL$).
Then,
\begin{itemize}
\item 
the support function~\( \sigma_{\Convex} \) is a norm on~$\DUAL$,
whose unit ball is the polar set~\( \Convex^{\odot} \), 
and \( \Convex^{\odot} \) is 
closed, convex, 
symmetric, weakly bounded and with full conical hull,
\item 
the support function~\( \sigma_{\Convex^{\odot}} \) is a norm on~$\PRIMAL$,
whose unit ball is \( \Convex \),
\item 
each norm is the dual norm of the other norm,
\item 
the topologies induced by both norms 
are both weaker than the Hilbertian topology.
\end{itemize}
The assertions remain true with ``weakly bounded'' replaced by ``bounded''
in the two instances where it appears. In that case, 
the topologies induced by both norms 
are equivalent to the Hilbertian topology.
  \label{pr:norm_and_dual_norm_Hilbert}
\end{proposition}

\begin{proof}
Being convex, the set~\( \Convex \) is closed in the weak topology,

By Proposition~\ref{pr:norm_and_dual_norm_dual_system},
the three first items hold true.
We use the property that 
the set~\( \Convex \) is closed in the weak topology,
and that the set~\( \Convex^{\odot} \) is closed in the weak topology,
hence is closed because it is convex (being a polar set).
\medskip

Regarding the fourth item, 
the topologies defined by the norm and by the dual norm
are both weaker (contain less open sets) 
than the Hilbertian topology, because, by construction, their
unit balls are closed (for the Hilbertian topology).
This results from Proposition~\ref{pr:comparison_of_norms}.
\medskip

If all the assumptions on \( \Convex \subset \PRIMAL \) are true, except for
weakly bounded replaced by bounded, then the three first items hold true 
because the bounded subset~\( \Convex \) is weakly bounded,
as seen in~\eqref{eq:barrier_cone_of_a_bounded_set}.
There remains to prove that the polar set~\( \Convex^{\odot} \) is bounded.
For this purpose, we denote by~\(\triplenorm{\cdot}\)
the norm \( \sigma_{\Convex} \) and we get 
    \begin{align*}
\BALL_{\triplenorm{\cdot}}=\Convex \mtext{ is closed}
&\Rightarrow 
0 \in \mathrm{int}\Convex=\mathrm{int}\BALL_{\triplenorm{\cdot}}
\tag{ by Proposition~\ref{pr:comparison_of_norms} }
\\
& \iff 
\exists m>0 \eqsepv
  m \BALL_{\norm{\cdot}} \subset \Convex 
\\
&\Rightarrow 
\Convex^{\odot} \subset \bp{m \BALL_{\norm{\cdot}}}^{\odot}
= \frac{1}{m}  \BALL_{\norm{\cdot}_\star}
\tag{ by~\eqref{eq:norm_and dual_norm_unit_ball} 
and the definition~\eqref{eq:polar_set} of a polar set }
\\
&\Rightarrow 
\Convex^{\odot} \subset \frac{1}{m}  \BALL_{\norm{\cdot}}
\tag{ because the dual norm~$\norm{\cdot}_\star$ of the Hilbertian norm
is the Hilbertian norm}
\\
&\Rightarrow 
\Convex^{\odot} \mtext{ is bounded.}
    \end{align*}
We conclude that the topologies induced by both norms 
are equivalent to the Hilbertian topology,
by Proposition~\ref{pr:equivalent_norms} 
because their balls are closed and bounded.

This ends the proof. 
\end{proof}

We provide assumptions under which a dual norm exists,
and we precise the topologies that norm and dual norm induce.

\begin{proposition}
  Let \( \triplenorm{\cdot} \) be a norm on a Hilbert space.
If the norm~\( \triplenorm{\cdot} \) 
is equivalent to the Hilbertian norm~\( \norm{\cdot} \) 
(or, equivalently,
if its unit ball \( \BALL_{\triplenorm{\cdot}} \) is bounded and closed), 
  then the norm \( \triplenorm{\cdot} \) 
  admits a dual norm~\( \triplenorm{\cdot}_\star \)
which is equivalent to the Hilbertian norm~\( \norm{\cdot} \) 
(or, equivalently 
whose unit ball~\( \BALL_{\triplenorm{\cdot}_\star} \) is 
bounded and closed), and 
\eqref{eq:norm_dual_norm_polar_balls_Hilbert}--\eqref{eq:norm_dual_norm=support_of_polar_balls_Hilbert}
hold true. 
\label{pr:dual_norm_Hilbert}
\end{proposition}

\begin{proof}
If the norm~\( \triplenorm{\cdot} \) 
  is equivalent to the Hilbertian norm~\( \norm{\cdot} \),
then the unit ball~\( \BALL_{\triplenorm{\cdot}} \) is closed and bounded
by Proposition~\ref{pr:equivalent_norms}.
Therefore, by Proposition~\ref{pr:norm_and_dual_norm_Hilbert}
with \( \Convex = \BALL_{\triplenorm{\cdot}} \), 
we deduce that the norm \( \triplenorm{\cdot} \) 
  admits a dual norm~\( \triplenorm{\cdot}_\star \),
that the topologies induced by both norms 
are equivalent to the Hilbertian topology,
and that \eqref{eq:norm_dual_norm_polar_balls_Hilbert}--\eqref{eq:norm_dual_norm=support_of_polar_balls_Hilbert}
hold true. 

  This ends the proof.
\end{proof}

\subsubsection{Dual norm in the Euclidian case}

\begin{proposition}
Any norm on~$\RR^d$ admits a dual norm.
   \label{pr:DUAL_NORM_Euclidian}  
\end{proposition}

\begin{proof}
We use Proposition~\ref{pr:dual_norm_Hilbert}, as all norms on~$\RR^d$
are equivalent to the Euclidian norm. 
\end{proof}

\newcommand{\noopsort}[1]{} \ifx\undefined\allcaps\def\allcaps#1{#1}\fi


\begin{thebibliography}{1}

\bibitem{Aliprantis-Border:1999}
C.~D. Aliprantis and K.~C. Border.
\newblock {\em Infinite dimensional analysis}.
\newblock Springer-Verlag, Berlin, second edition, 1999.

\bibitem{Argyriou-Foygel-Srebro:2012}
A.~Argyriou, R.~Foygel, and N.~Srebro.
\newblock Sparse prediction with the $k$-support norm.
\newblock In {\em Proceedings of the 25th International Conference on Neural
  Information Processing Systems - Volume 1}, NIPS'12, pages 1457--1465, USA,
  2012. Curran Associates Inc.

\bibitem{Bauschke-Combettes:2017}
H.~H. Bauschke and P.~L. Combettes.
\newblock {\em Convex analysis and monotone operator theory in {H}ilbert
  spaces}.
\newblock CMS Books in Mathematics/Ouvrages de Math\'ematiques de la SMC.
  Springer-Verlag, second edition, 2017.

\bibitem{Chancelier-DeLara:2019b}
J.-P. Chancelier and M.~{De Lara}.
\newblock A suitable conjugacy for the $l_0$ pseudonorm, 2019.
\newblock preprint.

\bibitem{Mirsky:1960}
L.~Mirsky.
\newblock {Symmetric Gauge Functions and Unitarily Invariant Norms}.
\newblock {\em The Quarterly Journal of Mathematics}, 11(1):50--59, 01 1960.

\bibitem{Moreau:1970}
J.~J. Moreau.
\newblock Inf-convolution, sous-additivit\'e, convexit\'e des fonctions
  num\'eriques.
\newblock {\em J. Math. Pures Appl. (9)}, 49:109--154, 1970.

\bibitem{Obozinski-Bach:hal-01412385}
G.~Obozinski and F.~Bach.
\newblock {A unified perspective on convex structured sparsity: Hierarchical,
  symmetric, submodular norms and beyond}.
\newblock working paper or preprint, Dec. 2016.

\bibitem{Obozinski-Jacob-Vert:inria-00628498}
G.~Obozinski, L.~Jacob, and J.-P. Vert.
\newblock {Group Lasso with Overlaps: the Latent Group Lasso approach}.
\newblock Research report, Oct. 2011.

\end{thebibliography}
\end{document}